\documentclass[11pt]{article}
\usepackage[utf8]{inputenc}
\usepackage{amsmath}
\usepackage{amsfonts}
\usepackage{amsthm}
\usepackage{amssymb}
\usepackage{mathtools}
\usepackage[T1]{fontenc}
\usepackage{pst-node}
\usepackage{tikz-cd}
\usepackage{caption}
\usepackage{bbm}
\usepackage{comment}
\usepackage{array}   
\newcolumntype{L}{>{$}l<{$}}
\usepackage{multirow}

\usetikzlibrary{calc,backgrounds}

\pgfdeclarelayer{background}
\pgfdeclarelayer{foreground}
\pgfsetlayers{background,main,foreground}

\usepackage{geometry}
\theoremstyle{plain}
\newtheorem{thm}{Theorem}[section]

\newtheorem{lem}[thm]{Lemma}
\newtheorem{prop}[thm]{Proposition}
\newtheorem{cor}[thm]{Corollary}

\newtheorem{ques}{Question}

\theoremstyle{definition}
\newtheorem{defn}[thm]{Definition}

\theoremstyle{remark}
\newtheorem*{rem}{Remark}

\newcommand{\re}{\mathbb{R}}
\newcommand{\co}{\mathbb{C}}

\newcommand{\ze}{\mathbb{Z}}

\newcommand{\na}{\mathbb{N}}
\newcommand{\Ker}{\mathrm{Ker}}
\renewcommand{\Im}{\mathrm{Im} \, }

\newcommand{\Ad}{\mathrm{Ad}}

\renewcommand{\cong}{\mathrm{\simeq}}

\newcommand{\ag}{\mathfrak{g}}
\newcommand{\ap}{\mathfrak{p}}
\newcommand{\au}{\mathfrak{u}}

\newcommand{\Cdim}{\mathrm{Confdim}}
\newcommand{\Hdim}{\mathrm{Hausdim}}

\newcommand{\SL}{\mathrm{SL}}

\renewcommand{\a}{\alpha}
\renewcommand{\b}{\beta}

\newcommand{\g}{\gamma}

\newcommand{\D}{\Delta}

\renewcommand{\l}{\lambda}

\newcommand{\s}{\sigma}
\renewcommand{\S}{\Sigma}
\renewcommand{\t}{\tau}
\renewcommand{\Tilde}{\widetilde}

\title{Vanishing of the second $L^p$-cohomology group for most semisimple groups of rank at least 3}

\author{Antonio López Neumann}

\date{}

\begin{document}

\maketitle

\begin{abstract}
    We show vanishing of the second $L^p$-cohomology group for most semisimple algebraic groups of rank at least 3 over local fields. More precisely, we show this result for $\SL(4)$, for simple groups of rank $\geq 4$ that are not of exceptional type or of type $D_4$ and for all semisimple, non-simple groups of rank $\geq 3$. Our methods work for large values of $p$ in the real case and for all $p>1$ in the non-Archimedean case. This result points towards a positive answer to Gromov's question on vanishing of $L^p$-cohomology of semisimple groups for all $p>1$ in degrees below the rank. The methods consist in using a spectral sequence à la Bourdon-Rémy, adapting a version of Mautner's phenomenon from Cornulier-Tessera and concluding thanks to a combinatorial case-by-case study of classical simple groups.

    \vspace*{2mm} \noindent{2020 Mathematics Subject Classification: } 20F67, 20G07, 22E41, 43A15.

    \vspace*{2mm} \noindent{Keywords and Phrases: } Semisimple algebraic groups, $L^p$-cohomology, spectral sequence, Heintze groups, root systems.
\end{abstract}

\section*{Introduction}


$L^p$-cohomology is a natural quasi-isometry invariant introduced first in \cite{GKS-Lp-cohomology} and popularized by Gromov in \cite{gromov}. It is a rather fine one in the sense that it gives quasi-isometry invariants for every $p>1$ and thus uncountably many (and a priori independent) quasi-isometry invariants. It can be defined in different settings: we may talk about simplicial $\ell^p$-cohomology for simplicial complexes, de Rham $L^p$-cohomology for manifolds, asymptotic $L^p$-cohomology for metric spaces or continuous group $L^p$-cohomology for locally compact second countable groups. Comparison theorems give criteria to guarantee when these different versions coincide (usually by comparing them to asymptotic $L^p$-cohomology).

In this article we deal with continuous group $L^p$-cohomology $H_\mathrm{ct}^*(G, L^p(G))$ of a locally compact second countable group $G$, endowed with a \textit{left}-invariant Haar measure, with coefficients in the \textit{right} regular representation on $L^p(G)$. See Section \ref{Section:continuous cohomology} for precise definitions. The main technical advantage of this version of $L^p$-cohomology is that we dispose of more algebraic machinery, namely, we may use the Hochschild-Serre spectral sequence for semi-direct products \cite[IX]{borel-wallach}.

The first locally compact groups we may probably think about are Lie groups. Often the term Lie group refers to real Lie groups, though we may also consider Lie groups over non-Archimedean local fields. Here we will use the terms \textit{simple} or \textit{semisimple group} to refer to a semisimple algebraic group over a local field (Archimedean or not). The main motivation for this article is the following question by Gromov: he predicts a classical behaviour of $L^p$-cohomology of semisimple groups for every $p>1$, in degrees at most equal to the rank \cite{gromov}.
\begin{ques}
Let $G$ be a semisimple group of rank $r \geq 2$ over a local field $F$. \\
$(1)$ Do we have $H_\mathrm{ct}^l(G, L^p(G)) = \{ 0  \}$ for all $l = 1, \ldots, r-1$ and $p >1$? \\
$(2)$ Do we have $H_\mathrm{ct}^r(G, L^p(G)) \neq \{ 0  \}$ at least for some values of $p$? \\
$(3)$ Is the space $H_\mathrm{ct}^r(G, L^p(G))$ Hausdorff for all $p > 1$?  
\end{ques}

Question $(2)$ has been addressed in \cite{bourdon-remy-non-vanishing} for real groups and \cite{lopez-top-degree} for non-Archimedean groups. 
This article deals with question $(1)$. Let us review some known results towards a positive answer to this question. Most of these results concern vanishing in degree 1 for all $p>1$ for groups of rank $\geq 2$. 

First, Pansu proved this result in the real case in some unpublished notes in 1999. In there, he showed a trichotomy for a homogeneous manifold $M$: either the isometry group of $M$ is a compact extension of a solvable unimodular real Lie group, or $M$ is quasi-isometric to a homogeneous space of strictly negative sectional curvature, otherwise $L^p H_\mathrm{dR}^1(M) = \{ 0 \}$ for all $p > 1$. Later Bader, Furman, Gelander and Monod proved vanishing of the first continuous cohomology of (real or non-Archimedean) simple groups of rank $\geq 2$ \textit{acting by isometries} on some $L^p(X, \mu)$, where $(X, \mu)$ is a standard Borel space and $p > 1$ \cite{BFGM}. In fact their result also applies to semisimple groups whose simple factors have rank $\geq 2$. Later, Cornulier and Tessera extended Pansu's trichotomy to semisimple groups over fields of characteristic zero \cite{cornulier-tessera} (their trichotomy is stated only for groups of characteristic zero but their argument for vanishing for semisimple groups of higher rank also works in positive characteristic). The present work is heavily influenced by \cite{cornulier-tessera}.

Another impressive contribution is the one of Lafforgue in \cite{lafforgue-renforcement-T}. Motivated by obtaining obstructions to adapt his own proof of the Baum-Connes isomorphism for hyperbolic groups to the case of $\SL(3, \ze)$, he defines a strong rigidity condition, known today as Lafforgue's Strong Property $(T)$, that (if proven for a large enough class of Banach spaces) implies an affirmative answer to Gromov's question in degree 1. He showed that simple connected algebraic groups containing $\SL(3, \re)$ have this property on Hilbert spaces \cite[2.1]{lafforgue-renforcement-T} and that groups containing $\SL(3, F)$, where $F$ is a non-Archimedean local field, have this property on all Banach spaces of type $>1$ \cite[0.3]{lafforgue-Fourier-rapide}. The main novelty for us is that this condition implies vanishing of the first continuous cohomology group for \textit{non-isometric} actions on Banach spaces (more precisely, for representations of small exponential growth on uniformly convex Banach spaces).

Bourdon and Rémy deal with vanishing of $L^p$-cohomology of real simple Lie groups in higher degrees \cite{bourdon-remy-vanishings}. They show that for some real simple Lie groups they call admissible, there are constants $p(k) > 1$ for every degree $k$ such that there is vanishing of $L^p$-cohomology in degree $k$ for all $1<p<p(k)$. Poincaré duality allows them to extend this result to large values of $p$, at least for large values of $k$ (in particular, this duality argument does not concern degrees below the rank). Their methods consist in proving a suitable version of the Hochschild-Serre spectral sequence for $L^p$-cohomology and invoke Pansu's description of the $L^p$-cohomology of the real hyperbolic space \cite{pansu08}. In fact their proof also applies to (non-admissible) real Lie groups, but the conditions on the constant $p(k)$ are much more restrictive (in this case the constant $p(k)$ tends to 1 for fixed $k$ when the rank goes to infinity).

Initially, we wanted to prove a statement similar to that of Bourdon and Rémy in the non-Archimedean case. It turned out that our methods also apply to the real case, but only for large values of $p$. The result we prove is the following.

\begin{thm}\label{Vanishing-in-degree-2-Intro} (see Theorems \ref{Vanishing for semisimple non-simple  groups} and \ref{Vanishing in degree 2} in the text)
Let $F$ be a local field and suppose that $G$ is either: \\
$\bullet$ $\SL(4, F)$, \\
$\bullet$ a simple group over $F$ of rank $r \geq 4$ that is not of type $D_4$ and is not of exceptional type, \\
$\bullet$ or a semisimple, non-simple group over $F$ of rank $r \geq 3$. \\
Then there exists a constant $p(G) \geq 1$ such that for all $p > p(G)$:
\begin{equation*}
    H_\mathrm{ct}^2(G, L^p(G)) = \{ 0  \}.
\end{equation*}
Moreover, when $F$ is non-Archimedean we have $p(G) = 1$.
\end{thm}

Combining our results with \cite[1.2]{lopez-top-degree}, we obtain a complete description of the vanishings of $L^p$-cohomology of $\SL(4, F)$, where $F$ is a non-Archimedean local field, for all $p>1$ and in all degrees. This gives a positive answer to Gromov's questions $(1)$ and $(2)$ for this group.

\begin{cor}
    Let $G = \SL(4,F)$, where $F$ is a non-Archimedean local field. We have for any $p>1$:
    \begin{equation*}
         H_\mathrm{ct}^k(G, L^p(G)) \neq \{ 0  \} \textit{ if and only if } k = 3.
    \end{equation*}
\end{cor}

For real admissible simple Lie groups (those for which the results in \cite{bourdon-remy-vanishings} apply) our result is complementary to that of \cite{bourdon-remy-vanishings}. In degree 2, their result gives vanishing for small values of $p$, say for $1 < p \leq q(G)$ for some constant $q(G)$. A priori, there is an interval $(q(G) , p(G)]$ for which none of our results show vanishing. The funny thing is that, for most admissible simple Lie groups, this interval disappears when the rank is large enough.

\begin{cor}\label{Uniform-Vanishing-for-Admissible-Groups} (see Corollary \ref{Uniform vanishing for admissible groups} in the text)
Let $G$ be an admissible real simple Lie group of rank $r \geq 8$ that is not of type $B_r$. Then we have for all $p > 1$
\begin{equation*}
    H_\mathrm{ct}^2(G, L^p(G)) = \{ 0  \}.
\end{equation*}
\end{cor}

We outline the proof of Theorem \ref{Vanishing-in-degree-2-Intro}. First, we use quasi-isometric invariance in order to identify $H_\mathrm{ct}^2(G, L^p(G))$ to $H_\mathrm{ct}^2(P, L^p(P))$, where $P$ is a maximal parabolic subgroup. This parabolic subgroup has a Levi decomposition $P = MSU$. The version of the Hochschild-Serre spectral sequence from \cite{bourdon-remy-vanishings} allows us to identify the space $H_\mathrm{ct}^2(P, L^p(P))$ (at least as an abstract vector space) to the first continuous cohomology group $H_\mathrm{ct}^1(M, L^p(M, V_p))$ of the Levi factor $M$ with coefficients in some Banach-valued $L^p$-space $L^p(M, V_p)$. The Banach space $V_p$ is in fact $H_\mathrm{ct}^1(SU, L^p(SU))$. The main technical problem comes from the fact that the continuous $M$-module $L^p(M, V_p)$ has relatively large exponential growth, forbidding us to (directly) invoke Lafforgue's strong property $(T)$ and obtain the desired vanishing. 

We amend this by passing to a cocompact, non-unimodular, solvable group $R$ and creating contractions thanks to its modular function $\D_R$. We then adapt techniques from \cite{cornulier-tessera} to show some non-isometric version of Mautner's phenomenon for the $R$-module $L^p(R, V_p)$. The upshot is that vanishing of $H_\mathrm{ct}^1(R, L^p(R, V_p))$ follows from the presence of enough contractions. Verifying this condition can be done easily in the semisimple, non-simple case (see Theorem \ref{Vanishing for semisimple non-simple  groups} in the text).

The non-trivial part consists in showing the existence of such contractions in the simple case. This is a battle between the exponential dilation of a maximal torus $A$ in $M$ acting on $V_p$ and the exponential contraction of $\D_R$ in certain directions. To show that the contraction of $\D_R$ wins this battle for the groups in the statement of Theorem \ref{Vanishing-in-degree-2-Intro}, we first control the operator norms of the action of A on $V_p$ by some term that can be written explicitly in combinatorial terms. This step requires large-scale geometric considerations, as our bound depends on the Hausdorff dimension of the Carnot-Carathéodory metric on $U$. Then, using the classification of semisimple algebraic groups over local fields, we reduce the existence problem of contracting elements to a combinatorial case-by-case study of root systems with multiplicities. The main point in the combinatorial part of the proof is that, for the infinite families of root systems ($A_r, B_r, C_r, BC_r$ and $D_r$) there exists always a choice of maximal parabolic subgroup such that our control of the exponential dilation grows linearly in the rank and the exponential contraction of the modular function grows quadratically in the rank (at least in well-chosen directions). This a priori asymptotic heuristic works quite fast: starting from $ r\geq 3$ for $A_r$, from $r\geq 4$ for $B_r, C_r$ and $BC_r$ and from $r \geq 5$ for $D_r$. Our current estimates do not seem to create contractions for exceptional groups.

We do not exclude the possibility that this result could be obtained using Lafforgue's Strong Property $(T)$. Indeed, Lafforgue extends his results from $M = \SL(3, F)$ to a higher rank semisimple group $G$ containing $M$ by restricting a representation of $G$ to $M$ and using that $M$ has strong property $(T)$. This does not optimize the constants in the exponential growth of the representation as they are the same as for $\SL(3,F)$. We expect that showing Lafforgue's strong property $(T)$ directly in $G$ improves the constants so that we can treat the representation appearing after applying the spectral sequence, at least starting from a certain rank. Nevertheless, our constants seem to be slightly better since our method already works when the Levi factor $M$ is $\SL(3,F)$ inside $G = \SL(4,F)$.


The article is organized as follows. Section 1 contains standard preliminaries concerning algebraic groups, continuous group cohomology and $L^p$-cohomology. We also recall the Hochschild-Serre spectral sequence for $L^p$-cohomology as given in \cite{bourdon-remy-vanishings}. Section 2 studies the first continuous cohomology group with coefficients in Banach-valued $L^p$-spaces. Using contraction arguments, we show a vanishing criterion for this space.
We directly apply this criterion to semisimple, non-simple groups of rank $\geq 3$. Sections 3 and 4 are the technical heart of this article. In Section 3 we study amenable hyperbolic groups with homotheties and obtain our concrete estimate for cohomology growth. Section 4 explains how to find contracting elements using root systems. We use the classification of simple algebraic groups over local fields, study in detail each group to search for contractions and sum up our results in tables. This proves Theorem \ref{Vanishing-in-degree-2-Intro}.
We also prove Corollary \ref{Uniform-Vanishing-for-Admissible-Groups} as a consequence of our combinatorial study.

\paragraph{Acknowledgements}

The present work is part of the author's PhD thesis.
I thank my advisors, Marc Bourdon and Bertrand Rémy, for their time, their constant support and regular communication. I also want to thank Mikael de la Salle for his attentive reading and for simplifying some arguments in Section 2, the current version of the article implements these simplifications.

\tableofcontents

\section{Continuous cohomology and algebraic groups}

This section establishes the setting of the article. We also fix notation for subsequent sections. It contains standard preliminaries of algebraic flavor, more precisely, algebraic groups and continuous cohomology. We also recall the spectral sequence for continuous group $L^p$-cohomology.

\subsection{Lie theoretic notions and Heintze groups}

These paragraphs collect the necessary Lie theoretic vocabulary and fix notations for subsequent sections. We follow standard references on algebraic groups, such as \cite{borel-algebraic-groups} or \cite[Chapter 0]{margulis-book}.

\paragraph{Roots and root spaces}

Let $G$ be a semisimple algebraic group over a local field $F$ (Archimedean or not) of split rank $r \geq 1$. Denote by $\ag$ its Lie algebra. Fix a maximal $F$-split torus $S$ in $G$ and denote by $X(S)$ the group of $F$-characters of $S$. Multiplication in $X(S)$ will be denoted additively. The \textit{restricted root system} $\Phi = \Phi_F(S, G)$ is the set of nonzero $F$-characters $\a \in X(S)$ such that the space
\begin{equation*}
    \ag_\a = \{ X \in \ag \, | \,\forall s \in S, \Ad(s) X = \a(s) X \}
\end{equation*}
is nonzero \cite[21.1]{borel-algebraic-groups}. An element $ \a \in \Phi$ is called a \textit{root relative to $F$} or a \textit{root}. The set $\Phi$ is a root system in the usual sense \cite[0.26]{margulis-book}. The integer $\dim \ag_\a$ will be called the \textit{multiplicity} of the root $\a$, and will be denoted $m(\a)$.

Recall that the root system $\Phi$ is said to be \textit{reduced} if for all $\a \in \Phi$, the element $2\a$ is not a root. The only family of non-reduced root systems consists of root systems of type $BC_n$ for $n\geq 2$. The space $\ag_\a$ is a commutative Lie subalgebra when $2\a \notin \Phi$ (this follows from $[\ag_\a, \ag_\b] \subseteq \ag_{\a + \b}$), but it may not be a Lie subalgebra when $\a, 2 \a \in \Phi$. We amend this by considering the space $\ag_{(\a)} : = \ag_\a \oplus \ag_{2\a}$, which is always a Lie subalgebra of $\ag$ for $\a \in \Phi$ \cite[21.7]{borel-algebraic-groups}. If we define $\Phi_{nd}$ to be the set of roots $\a \in \Phi$ such that $\a / 2$ is not a root, then the \textit{root space decomposition} of $\ag$ may be written:
\begin{equation*}
    \ag = \ag_0 \oplus \bigoplus_{\a \in \Phi_{nd}} \ag_{(\a)}.
\end{equation*}
where each summand is a Lie subalgebra \cite[21.7]{borel-algebraic-groups}.

Let $U_{(\a)}$ be the unique unipotent closed Zariski-connected Lie subgroup with Lie algebra $\ag_{(\a)}$ normalized by the centralizer $Z(S)$ of the torus $S$ \cite[21.9 (i)]{borel-algebraic-groups}. 

If $ \a \in \Phi$ and $2 \a \notin \Phi$, then $ U_\a : = U_{(\a)}$ is abelian and $F$-isomorphic (as an algebraic group) to $\ag_\a$ \cite[21.20]{borel-algebraic-groups}. If $\theta_\a :\ag_\a \to U_{\a}$ denotes such an isomorphism, then the action by conjugation of some $s \in S$ on $U_\a$ becomes the homothety of ratio $\a(s)$ on $\ag_\a$, that is: $s \theta_\a(X) s^{-1} = \theta_\a ( \a(s) X) $ for all $X \in \ag_\a$ \cite[3.17]{borel-tits-IHES}.

If $\a , 2\a \in \Phi$ then $U_{(\a)}$ is metabelian, its center is $U_{2\a}$ \cite[21.10 (2)]{borel-algebraic-groups} and is $F$-isomorphic (as a variety) to the product $U_{(\a)} / U_{2\a} \times U_{2\a}$ \cite[21.20 Proof of (i)]{borel-algebraic-groups}. The quotient $U_\a : = U_{(\a)} / U_{2\a}$ is $F$-isomorphic (as a variety) to $\ag_\a$. If $\theta_\a :\ag_\a \to U_{\a}$ denotes this isomorphism, then the action by conjugation of some $s \in S$ on $U_\a$ becomes again the homothety of ratio $\a(s)$ on $\ag_\a$ \cite[3.17]{borel-tits-IHES}. Notice that under these identifications, the action of $s \in S$ on $U_{(\a)}$, seen as the product $U_\a \times U_{2\a}$, is a homothety of ratio $\a(s)$ on $U_\a$ and a homothety of ratio $2\a(s)$ on $U_{2\a}$.

In the real case, the isomorphism $\theta_\a$ is just the exponential mapping.

\paragraph{Parabolic subgroups and Levi decomposition}

We fix a choice of simple roots $\S$ inside $\Phi$, denote by $\Phi^+$ the corresponding positive roots. For any $\a \in \Phi$ and $\s \in \S$, set $n_\s(\a) \in \ze$ so that 
\begin{equation*}
    \a = \sum_{\s \in \S} n_\s(\a) \s.
\end{equation*}
For $I \subseteq \S$, denote by $\Phi_I$ the set of roots in $\Phi$ which are linear combinations of simple roots in $I$ ($\Phi_I$ is a root system in its own right) and $\Phi_I^+ = \Phi_I \cap \Phi^+$. In what follows we will introduce many notations for subgroups of $G$ depending on a subset $I$ of $\S$. If $I$ is in subscript it means that the corresponding subgroup contains in some way the roots in $I$, if $I$ is in superscript it means that the corresponding subgroup avoids in some way the roots in $I$.

For $I \subseteq \S$, the \textit{standard parabolic subgroup} $P_I$ of type $I$ is the subgroup of $G$ with Lie algebra 
\begin{equation*}
    \ap_I = \ag_0 \oplus \bigoplus_{\a \in \Phi_I} \ag_\a \oplus \bigoplus_{\a \in \Phi^+ \setminus \Phi_I^+} \ag_\a
\end{equation*}
\cite[4.2]{borel-tits-IHES}. Let $S^I$ denote the $F$-split subtorus of $S$ of rank $(r - |I|)$ defined by  $ S^{I}= (\bigcap_{\g \in I} \Ker \, \g)^{0}$. 
The group $P_I$ admits a \textit{Levi decomposition} $P_I = Z_I \ltimes U^I$, where $Z_I = Z_G(S^{I})$ denotes the centralizer of the torus $S^{I}$ in $G$ and $U^I$ is the unipotent radical of $P_I$ \cite[4.2]{borel-tits-IHES}. We will briefly present some properties of both factors in this decomposition.

The group $Z_I = Z_G(S^{I})$ is called the \textit{Levi factor} of $P_I$ and has Lie algebra $\ag_0 \oplus \bigoplus_{\a \in \Phi_I} \ag_\a$. It is reductive \cite[2.15 d)]{borel-tits-IHES}, and so it may be decomposed as the almost direct product $M_I T^{I}$ where $M_I = [ Z_I, Z_I ]$ is a semisimple group of split rank $|I|$ and $T^{I}$ is the connected center of $Z_I$. The group $T^I$ is defined over $F$ \cite[2.15 a)]{borel-tits-IHES} and is the almost direct product of $S^I$ with a (compact) anisotropic subtorus defined over $F$ \cite[1.8]{borel-tits-IHES}. The decomposition $P_I = (M_I T^{I}) \ltimes U^{I}$ is sometimes called the \textit{Langlands decomposition} of $P_I$. 

\begin{rem}
The action by conjugation of $M_I$ on $U^{I}$ preserves the volume of $U^{I}$. Indeed, volume dilation of this action defines a character of $M_I$, but $M_I$ has no nontrivial characters because it is a semisimple group.
\end{rem}

On the other hand, the group $U^{I}$ has Lie algebra $ \au^{I} = \bigoplus_{\a \in \Phi^+ \setminus \Phi_I^+} \ag_\a$ and is $F$-isomorphic (as a variety) to the direct product of the corresponding root groups \cite[21.9 (ii)]{borel-algebraic-groups}:
\begin{equation*}
    \prod_{\a \in \Phi_{nd}^+ \setminus (\Phi_I)_{nd}^+} U_{(\a)} = \prod_{\a \in \Phi^+ \setminus \Phi_I^+} U_{\a}.
\end{equation*}



\paragraph{Heintze groups and Iwasawa decomposition}

We are mostly concerned with (proper) \textit{maximal} parabolic subgroups, that is, the case where $I = \S \setminus \{ \g\}$ for some $\g \in \S$. In this case, $ S_\g : = S^I \cong F^*$ and $\Phi_{n_\g > 0} : = \Phi^+ \setminus \Phi_{\S \setminus \{ \g\}}^+$ is exactly the set $\{ \a \in \Phi, n_\g(\a) > 0 \}$. Conjugation by some $s \in S_\g$ on $x \in U_\a$, where $x =\theta_\a(X)$ and $\theta_\a$ is our $S$-equivariant identification of $U_\a$ with $\ag_\a$, becomes:
\begin{equation*}
    s. x  : = s  x  s^{-1} = \theta_\a(\g(s)^{n_\g(\a) }X).
\end{equation*}
Thus a fixed $s \in S_\g$ either contracts or dilates all the root groups $U_\a$ for $\a \in \Phi_{n_\g > 0}$ (depending on the sign of $\log |\g(s)|$). In the terminology of \cite{caprace-cornulier-tessera-monod}, $S_\g$ acts on $U_\g : = U^{I}$ by confining automorphisms (the notation $U_\g$ is ambiguous as it may also refer to the root group $U_\g$ integrating $\ag_\g$, but in practice we will never use this notation in this sense).
\begin{defn}
    The group $H_\g := S_\g \ltimes U_\g$ is the \textit{Heintze group} associated to the simple root $\g$.
\end{defn}

By \cite[4.6]{caprace-cornulier-tessera-monod}, the solvable group $H_\g$ is an amenable, non-unimodular, Gromov-hyperbolic locally compact group.

On the other hand, the semisimple part $M^\g : = M_{I}$ of the group $P_I$ admits an \textit{Iwasawa decomposition} $K^\g A^\g N^\g$, where $K^\g$ is a maximal compact subgroup of $M^\g$, the group $A^\g$ is a maximal $F$-split torus in $M^\g$ (hence a group-theoretic supplementary of $S_\g$ in $S$) and $N^{\g}$ is isomorphic to the product $\prod_{\a \in \Phi_{I}^+} U_\a$ and has Lie algebra $\bigoplus_{\a \in \Phi_{I}^+} \ag_\a$ (see \cite[IX. 1.3]{helgason} for the real case and \cite[2.6.11]{macdonald-spherical-functions} for the non-Archimedean case).

\begin{defn}
We write $R^\g := A^\g N^\g$.
\end{defn}

The main object of study of the subsequent sections is the semi-direct product $R^\g \ltimes H_\g$. By this we mean that we will study the group $H_\g$ and the action by conjugation of the group $R^\g$ on $H_\g$. 
To sum up the relations between the different groups in this section:
\begin{equation*}
    G \simeq_\mathrm{qi} P_I  \simeq_\mathrm{qi} M^\g \ltimes H_\g \simeq_\mathrm{qi} R^\g \ltimes H_\g,
\end{equation*}
where the metrics on each group are just the metrics induced from $G$. Moreover, all of these quasi-isometries are cocompact inclusions.

\subsection{Continuous group cohomology}\label{Section:continuous cohomology}

We define continuous cohomology of a locally compact second countable group following \cite[IX]{borel-wallach}.

Let $G$ be a locally compact second countable group. Then $G$ is a countable union of compact sets (for instance, because $G$ carries a proper metric \cite{struble}). Let $(\rho, V)$ be a continuous representation of $G$ (we also use continuous $G$-module as terminology) i.e. a morphism $\rho : G \to B(V)$ such that the map $G \times V \to V$ defined by $(g,v) \mapsto \rho(g)v$ is continuous, where $V$ is some locally convex topological vector space and $B(V)$ denotes continuous invertible operators on $V$. Here $V$ will always be at least a Fréchet space.

For $k \in \na$, we define the space $C^k(G, V)$ of $k$-cochains as the set of continuous maps from $G^{k+1}$ to $V$. Since $G$ is a countable union of compact sets, the space $C^k(G, V)$ equipped with the topology of uniform convergence on compact subsets is a Fréchet space. 

We define the differential $d_k : C^k(G, V) \to C^{k+1}(G, V)$ of a $k$-cochain $c$ as:
\begin{equation*}
    (d_k c)(g_0, \ldots, g_{k+1}) = \sum_{i = 0}^{k+1} (-1)^{i}c(g_0, \ldots, g_{i-1}, g_{i+1}, \ldots, g_{k+1}).
\end{equation*}
It satisfies $d_{k+1} \circ d_k = 0$.

The space $C^k(G, V)$ can be viewed as a continuous $G$-module, by endowing it with the action: \begin{equation*}
    (g.c)(g_0, \ldots, g_k) = \rho(g)( c( g^{-1}g_0, \ldots, g^{-1}g_k)).
\end{equation*}

We consider the space $ C^k(G, V)^G$ of invariants in $ C^k(G, V)$ with respect to this action. Notice that this is just the set of continuous $G$-equivariant maps from $G^{k+1}$ to $V$, when endowing $G^{k+1}$ with the diagonal action by left translation on each factor and $V$ with the action given by $\rho$. The differential $d_k$ maps $ C^k(G, V)^G$ into $C^{k+1}(G, V)^G$. We call $\ker( d_k |_{ C^k(G, V)^G})$ the space of \textit{$k$-cocycles} and denote it by $Z^k(G,\rho)$, we call $\Im ( d_{k-1}|_{C^{k-1}(G, V)^G})$ the space of \textit{$k$-coboundaries} and denote it $B^k(G,\rho)$. 

In homological algebra, group cohomology with values in a representation is defined as the cohomology of the complex of invariants of any resolution of the given representation. In this article we will not use other resolutions apart from the one we already defined (even though we invoke a result by \cite{bourdon-remy-vanishings} obtained using other resolutions), so the following definition of continuous group cohomology is enough for our purposes.

\begin{defn}
The \textit{$k$-th continuous cohomology group} (resp. \textit{$k$-th reduced continuous cohomology group}) of $G$ with coefficients in $(\rho,V)$ is the topological vector space:
\begin{equation*}
    H^k_\mathrm{ct} (G, \rho) := Z^k(G,\rho) / B^k(G,\rho) \quad ( \textrm{resp. }  \overline{H}^k_\mathrm{ct} (G, \rho) := Z^k(G,\rho) / \overline{B^k(G,\rho)} ).
\end{equation*}
\end{defn}
The space $\overline{H}^k_\mathrm{ct} (G, \rho)$ is the biggest Hausdorff quotient of $H^k_\mathrm{ct} (G, \rho)$. In particular these two spaces coincide exactly when $H^k_\mathrm{ct} (G, \rho)$ is Hausdorff.

\paragraph{Non-homogeneous cochains}

For concrete applications and particularly in degree 1, it is sometimes useful to view elements of $C^k(G,V)^G$ not as maps from $G^{k+1}$ to $V$, but as maps from $G^k$ to $V$. 

For $k= 1$, this gives the classical geometric interpretation of the first continuous cohomology group. We can identify $Z^1(G,\rho)$ with the space of continuous maps $b : G \to V$ satisfying the cocycle relation $b(gh) = b(g) + \rho(g)b(h)$. This space can in turn be identified with the space of continuous affine actions of $G$ on $V$ with linear part $\rho$, via the map $b \mapsto A_b$ for $b \in Z^1(G,\rho)$, where $A_b(g): V \to V,  v \mapsto \rho(g) v + b(g)$ for $g \in G$. In a similar way, we may identify $B^1(G,\rho)$ with the space of maps of the form $b(g) = v - \rho(g) v$ for some $v \in V$. These maps correspond to continuous affine actions of $G$ on $V$ having a fixed point. If the group $G$ is compactly generated, $S$ is a compact generating set and $(V, || \cdot ||_V)$ is a Banach space, then $Z^1(G,\rho)$ is a Banach space with norm $||b|| =  \sup_{x \in S} ||b(x)||_V$ for $b \in Z^1(G,\rho)$.

\paragraph{Continuous group $L^p$-cohomology}

Let $G$ be a locally compact second countable group endowed with a left-invariant Haar measure $\mu_G$.
In this article we will be interested in the Banach space $V = L^p(G)$ of $p$-integrable functions with respect to the measure $\mu_G$ for $1 <p < \infty$. The representation $\rho_p$ of $V$ we are interested in is the \textit{right regular representation} of $G$, defined by right translation on the argument of an $L^p$-function: $\rho_p(g) f (x) = f (xg)$. 
This defines a continuous representation of $G$. Notice that it is an isometric representation if and only if the measure $\mu_G$ is also right-invariant, that is, if and only if $G$ is unimodular. In this case $\rho_p$ is also continuously conjugate to the similarly defined left regular representation $\l_p$ via the continuous linear map $L^p(G) \to L^p(G)$ sending $f$ to $\Check{f}(x) = f (x^{-1})$.

The representation $\rho_p$ will appear for several different groups, we will not need to specify the action as it will be enough to specify the corresponding vector space. In order to avoid (even more) cumbersome notation, $H^k_\mathrm{ct} (G, L^p(G))$ will just mean $H^k_\mathrm{ct} (G, \rho_p)$.

The most important feature of $L^p$-cohomology, in contrast to cohomology with respect to an arbitrary representation, is that it is a quasi-isometry invariant for any degree $k \in \na$ and $1 < p < \infty$. This is a phenomenon that has been shown for different types of $L^p$-cohomology by different people: see \cite[p. 219]{gromov} for a sketch of proof in the simplicial $\ell^p$-cohomology setting, \cite{bourdon-pajot} for a more detailed proof in the same context, \cite{pansu95} for a proof for de Rham $L^p$-cohomology. Here we are interested in quasi-isometric invariance of continuous group $L^p$-cohomology. It was first proven by Elek in \cite{elek-coarse} for finitely generated groups by comparing it with asymptotic $L^p$-cohomology. In \cite{sauer-schrodl}, the same idea is used to show that vanishing of an $\ell^2$-Betti number is a coarse equivalence invariant. In \cite{bourdon-remy-vanishings}, quasi-isometric invariance of continuous group $L^p$-cohomology is proven for general locally compact second countable groups.

\begin{thm}\cite[Theorem 1.1]{bourdon-remy-vanishings}
Let $G_1$ and $G_2$ be locally compact second countable groups, equipped with left-invariant proper metrics $d_1$ and $d_2$. Every quasi-isometry $F : (G_1, d_1) \to (G_2, d_2)$ induces canonically an isomorphism of graded topological vector spaces:
\begin{equation*}
    F^*: H^*_\mathrm{ct} (G_2, L^p(G_2)) \to  H^*_\mathrm{ct} (G_1, L^p(G_1))
\end{equation*}
for every $ 1 < p < \infty$. The same holds for reduced cohomology.
\end{thm}

\begin{rem}
1. The idea of both \cite{sauer-schrodl} and \cite{bourdon-remy-vanishings} consists in comparing continuous group $L^p$-cohomology to asymptotic $L^p$-cohomology. The latter is a coarse equivalence invariant \cite{pansu95} and hence, even if not stated explicitly, continuous group $L^p$-cohomology is also invariant under coarse equivalences.

2. Continuous cohomology with coefficients in the left regular representation $\l_{p,G}$ is not invariant under quasi-isometries when at least one of the two groups is not unimodular (when both are unimodular, it is same as considering the right regular representation). For instance, the groups $G = \SL(2, \re)$ and $B < \SL_2(\re)$ of upper triangular matrices are quasi-isometric. One has $ \overline{H}^1_\mathrm{ct} (G, \l_{p,G}) \neq \{ 0 \}$ for all $p > 1$ ($G$ is unimodular, so $\overline{H}^1_\mathrm{ct} (G, \l_{p,G}) =\overline{H}^1_\mathrm{ct} (G, \rho_{p,G})$ and $G$ is quasi-isometric to the real hyperbolic plane $\mathbb{H}^2_\re$, therefore using comparison theorems between continuous group $L^p$-cohomology, asymptotic $L^p$-cohomology and de Rham $L^p$-cohomology \cite[Theorems 6.5 and 6.7]{bourdon-remy-non-vanishing} we have $\overline{H}^1_\mathrm{ct} (G, \rho_{p,G}) = L^p H_\mathrm{dR}^1(\mathbb{H}^2_\re)$ and the latter is nonzero for all $p>1$ \cite[5.2]{pansu89}). On the other hand $ \overline{H}^1_\mathrm{ct} (B, \l_{2,B}) = \{ 0 \}$ \cite{delorme-1-cohomologie-reps-unitaires}.
\end{rem}

\subsection{Spectral sequence reduction}

Let $G$ be a semisimple group over a local field $F$ of split rank $r$.
The idea of this section is to use quasi-isometric invariance of continuous group $L^p$-cohomology \cite[Theorem 1.1]{bourdon-remy-vanishings} and the Hochschild-Serre spectral sequence. This reduces the computation of the second $L^p$-cohomology group of $G$ to the first continuous group cohomology of a semisimple factor of a well-chosen parabolic subgroup with values in a more exotic $L^p$-module. The following version of the spectral sequence is well-suited to compute $L^p$-cohomology:

\begin{thm}\cite[Corollary 5.4]{bourdon-remy-vanishings} \label{Lp-Hochschild-Serre}
Let $P$ be a locally compact second countable group. Suppose that $P = Q \ltimes S$ where $Q$ and $S$ are two closed subgroups of $P$, such that $C^*(S, L^p(S))$ is homotopically equivalent to a complex of Banach spaces and every cohomology space $H^k_\mathrm{ct}(S, L^p(S))$ is Hausdorff. Then there exists a spectral sequence $(E_r)$ abutting to $H^*_\mathrm{ct}(P, L^p(P))$ in which:
\begin{equation*}
    E_2^{k,l} = H^k_\mathrm{ct} (Q, L^p(Q, H^l_\mathrm{ct}(S, L^p(S)))).
\end{equation*}
\end{thm}

We keep the same notations as in the previous section. Let $\g \in \S$ be a simple root of $\Phi = \Phi(S, G)$. The group $G$ is quasi-isometric to the semi-direct product $R^\g \ltimes H_\g$. The spectral sequence will turn out to be useful on this semi-direct product because the $L^p$-cohomology of $H_\g$ is sufficiently well-understood.

\begin{cor}\label{Hk(G,Lp(G))=Hk-1(R,Lp(R,H1(AN,Lp(AN))))}
For any simple root $\g \in \S$, any $p > \max \{ \mathrm{Confdim}(\partial H_\g), 1 \}$ and any integer $k \geq 1$, we have linear isomorphisms:
\begin{equation*}
    H_\mathrm{ct}^k(G, L^p(G)) = H_\mathrm{ct}^{k-1}(R^\g, L^p(R^\g, H_\mathrm{ct}^1(H_\g, L^p(H_\g)))).
\end{equation*}
The $R^\g$-action of $g \in R^\g$ on $b : H_\g \to L^p(H_\g)$ is defined by:
\begin{equation*}
    (\pi_0(g) b) (h) (x) = b (g^{-1}h g)(g^{-1}x g).
\end{equation*}
The $R^\g$-action $\pi$ on $f \in L^p(R^\g, H_\mathrm{ct}^1(H_\g, L^p(H_\g)))$ is defined for $g,h \in R^\g$ by:
\begin{equation*}
    (\pi(g) f)(h) =\pi_0(g)(f(hg)).
\end{equation*}
\end{cor}

\begin{proof} Quasi-isometric invariance of continuous group $L^p$-cohomology \cite[1.1]{bourdon-remy-vanishings} gives topological isomorphisms:
\begin{equation*}
    H_\mathrm{ct}^k(G, L^p(G)) = H_\mathrm{ct}^k(R^\g \ltimes H_\g, L^p(R^\g \ltimes H_\g)).
\end{equation*}
We apply Theorem \ref{Lp-Hochschild-Serre} to the semi-direct product $R^\g \ltimes H_\g$. When $F$ is non-Archimedean, the group $H_\g$ is quasi-isometric to a tree \cite[4.6]{cornulier-tessera}, so its $L^p$-cohomology is Hausdorff and concentrated in degree 1 for all $p>1$. In the real case, if $p > \mathrm{Confdim}(\partial H_\g)$ then $H_\mathrm{ct}^{1}(H_\g, L^p(H_\g)) \neq  \{ 0 \}$ \cite[Theorem 1]{cornulier-tessera}, is Hausdorff \cite[11.9]{tessera-sobolev} and moreover $H_\mathrm{ct}^{k}(H_\g, L^p(H_\g)) =  \{ 0 \}$ for $k \geq 2$ \cite[Corollaire B]{bourdon-lp-degre-superieur}. Thus the Hochschild-Serre spectral sequence collapses in the $E_2$ page and gives the desired linear isomorphisms.
\end{proof}

\begin{rem}
This is the only part of the proof where we need $p$ to be large in the real case. The rest of the proof works uniformly for all $p>1$, both in the real and in the non-Archimedean case.
If one would like to study the second $L^p$-cohomology group of real simple Lie groups for smaller values of $p$, the spectral sequence will give us other isomorphisms. For instance, if $H_\g$ is the real hyperbolic space of dimension $d$, Bourdon and Rémy use \cite{pansu08} and obtain vanishing for $p \leq (d-1)/2$ \cite[1.4]{bourdon-remy-vanishings}, and for $ (d-1)/2 <p \leq d-1 $, we obtain 
\begin{equation*}
H_\mathrm{ct}^2(G, L^p(G)) = H_\mathrm{ct}^{0}(R^\g, L^p(R^\g, H_\mathrm{ct}^2(H_\g, L^p(H_\g)))) =  L^p(R^\g, L^pH_\mathrm{dR}^2(H_\g))^{R^\g}.    
\end{equation*}
Adapting the techniques we will introduce further on continuous $L^p$-cohomology to de Rham $L^p$-cohomology may also show that this space vanishes for many groups. We do not do it because in practice, most of the vanishings obtained in this manner are contained in Corollary \ref{Uniform-Vanishing-for-Admissible-Groups} (except for type $B_r$ and some low rank cases).
\end{rem}

\section{Contracting automorphisms on Banach-valued $L^p$-spaces}

In this section, we adapt techniques from \cite[Section 2]{cornulier-tessera} to show vanishing of the right hand side of Corollary \ref{Hk(G,Lp(G))=Hk-1(R,Lp(R,H1(AN,Lp(AN))))} when $k = 2$, under contraction hypotheses that will be verified in subsequent sections.

The setting for the section is the following (except for the Lie theoretic statements at the very end). Let $G$ be a locally compact second countable group endowed with a left-invariant Haar measure $\mu$ and $(\pi_0 , V)$ be a continuous $G$-module on some separable Banach space $V$. Let $\pi$ denote the action on measurable functions $f : G \to V$ defined by:
\begin{equation*}
    (\pi(g) f)(h) = \pi_0(g)(f(hg)),
\end{equation*}
for all $g, h \in G$. For $p > 1$, we consider the Banach space $L^p(G, V)$ of Bochner $p$-integrable functions, that is, the set of measurable functions $f : G \to V$ such that:
\begin{equation*}
    ||f||_p^p = \int_G ||f(g)||_V^p \, \mathrm{d}\mu(g) < \infty.
\end{equation*}
We may also denote by $\pi_p$ the restriction of $\pi$ to the space $L^p(G,V)$.

\subsection{Operator norms and the modular function}

Our definition of the modular function $\D_G$ of $G$ is given by the following formula: for any measurable set $E$ of $G$ and $g \in G$ we have $\mu(Eg) = \D_G(g) \mu(E)$, or alternatively, for any continuous compactly supported function $\phi$ on $G$ we have:
\begin{equation*}
    \int_G \phi(hg^{-1}) \mathrm{d} \mu(h) = \D_G(g) \int_G \phi(h) \mathrm{d}\mu(h).
\end{equation*}

If $(E, || \cdot ||_E)$ is any Banach space and $A : E \to E$ is some bounded linear operator, we define the operator norm of $A$ by:
\begin{equation*}
    ||| A|||_{E} = \sup_{v \in E, v \neq 0} \frac{|| A v||_E}{|| v||_E}.
\end{equation*}

We first obtain the following identity on the operator norms of $\pi$. It is central in our reasoning because it shows that even though the operator norms of $\pi_0$ can be really big, we can hope to compensate them using the modular function.

\begin{prop}\label{generalGrowthEstimate}
For all $g \in G$ we have:
\begin{equation*}
    |||\pi(g)|||_{L^p(G,V)} = \Delta_G(g)^{-1/p} |||\pi_0(g)|||_V.
\end{equation*}
\end{prop}

\begin{proof}
The proof follows from the definition of the modular function:
\begin{align*}
    ||\pi(g)f||_{L^p(G,V)}^p & = \int_G ||\pi_0(g)( f(hg))||_V^p \, \mathrm{d}\mu(h) \\
    & = \Delta_G(g)^{-1} \int_G ||\pi_0(g)( f(h))||_V^p \, \mathrm{d}\mu(h) \\
    & \leq \Delta_G(g)^{-1} |||\pi_0(g)|||_V^p ||f||_{L^p(G,V)}^p
\end{align*}
For the reverse inequality, fix $\varepsilon > 0$ and let $v \in V$ be such that $||\pi_0(g) v||_V \geq (|||\pi_0(g)||| - \varepsilon) ||v||_V$.
Fix some compact $K \subset G$ such that $\mu_G(K) = 1$ and we may define $f_v \in L^p(G,V)$ by $f_v(x) = v$ for $x \in K$ and $f(x) = 0$ for $x \in G \setminus K$. We have $||f_v||_{L^p(G,V)} = ||v||_V$ and
\begin{align*}
    ||\pi(g)f_v||_{L^p(G,V)}^p & = \Delta_G(g)^{-1} \int_K ||\pi_0(g)v||_V^p \, \mathrm{d}\mu(h) \\
    & \geq \Delta_G(g)^{-1} (|||\pi_0(g)||| - \varepsilon)^p ||f_v||_{L^p(G,V)}^p.
\end{align*}
\end{proof}

\subsection{Contracting automorphisms and cohomology}

We introduce the main object of study of this section, namely, $\pi_p$-contracting elements. This subsection is devoted to showing some complementary results that are not necessary for the proof of Theorem \ref{Vanishing-in-degree-2-Intro} (though we will need Lemma \ref{Pi-bounded-invariance} for some cases) but that highlight the importance of contractions.

\begin{defn}
We say that $\xi \in G$ is \textit{$\pi_p$-contracting} if 
\begin{equation*}
    \lim_{n \to +\infty} ||| \pi(\xi^n)|||_{L^p(G, V)} = 0,
\end{equation*} 
and \textit{$\pi_p$-bounded} if $\xi$ generates a non-relatively compact semigroup and \begin{equation*}
    \sup_{n>0} ||| \pi(\xi^n)|||_{L^p(G, V)} < \infty.
\end{equation*}
\end{defn}

A non-compact group $G$ satisfies a linear Sobolev inequality on $L^p(G)$ if and only if $G$ is not simultaneously amenable and unimodular. 
The first result we show is a generalization of the fact that a non-unimodular group satisfies a linear Sobolev inequality on $L^p(G)$ \cite[11.9]{tessera-sobolev}, but this time for the representation $\pi_p$.

\begin{lem}\cite[11.10]{tessera-sobolev}
Suppose that there exists a $\pi_p$-contracting element $\xi \in G$. Then the space $H_\mathrm{ct}^1(G, \pi_p)$ is Hausdorff, that is:
\begin{equation*}
    H_\mathrm{ct}^1(G, \pi_p ) = \overline{H}_\mathrm{ct}^1(G, \pi_p ).
\end{equation*}
\end{lem}

\begin{proof}
Up to changing $\xi$ by some power, we may suppose that $||| \pi(\xi)|||_{L^p(G, V)} < 1$. Since $G$ is $\s$-compact, we may write it as an increasing countable union of compact subsets $(Q_k)_k$. For $k$ large enough we have that $\xi \in Q_k$ and hence for $f \in L^p(G,V)$:
\begin{equation*}
    ||f||_{p, Q_k} = \sup_{g \in Q_k} ||f - \pi(g) f ||_p \geq ||f - \pi(\xi) f ||_p \geq ||f||_p - ||\pi(\xi)f||_p \geq C ||f||_p
\end{equation*}
where $C = 1 - |||\pi(\xi)|||_{L^p(G, V)} > 0$. Therefore $B^1(G, \pi_p)$ is closed in $Z^1(G, \pi_p)$.
\end{proof}

\begin{lem}\label{Pi-bounded-invariance} \cite[2.2]{cornulier-tessera}
Suppose that there exists a $\pi_p$-bounded element $\xi \in G$. If $f \in L^p(G, V)$ is $\pi(\xi)$-invariant, then $f = 0$.
\end{lem}

\begin{proof}
 Let $X$ be any compact subset of $G$ and denote by $||f||_{X,p}$ the $L^p$-norm of $f 1_X$. Since $\xi$ generates a non-relatively compact semigroup and the action of $G$ on itself is proper, we may take a subsequence $(n_k)_{k \in \na}$ so that the translates $(X\xi^{n_k})_{k>0}$ are disjoint. Since $\xi$ is $\pi$-bounded, let $C = \sup_{n>0} ||| \pi(\xi^n)|||_{L^p(G, V)} < \infty$. We have:
\begin{align*}
    ||f||_{X,p}^p&  = ||\pi(\xi^n) f||_{X,p}^p = \int_X ||\pi_0(\xi^n) (f(h \xi^n)) ||_{V}^p \, \mathrm{d} \mu(h)\\ 
    &= \Delta_G(\xi)^{-n} \int_{X\xi^n} ||\pi_0(\xi^n) (f(h) ) ||_{V}^p \, \mathrm{d} \mu(h) \\
    & \leq \Delta_G(\xi)^{-n} |||\pi_0(\xi^n)|||_V^p \int_{X\xi^n} ||f(h)||_{V}^p \, \mathrm{d} \mu(h) \\
    & = |||\pi(\xi^{n})|||_{L^p(G,V)}^p ||f||_{X\xi^n, p}^p \leq C ||f||_{X\xi^n, p}^p.
\end{align*}
But $\sum_{k>0} ||f||_{X\xi^{n_k}, p} \leq ||f||_{G,p} < \infty$, hence $||f||_{X\xi^{n_k}, p} \to 0$ when $k \to +\infty$. Therefore the previous inequality implies that $||f||_{X,p} = 0$ for any compact $X \subset G$, which gives $f = 0$ almost everywhere.
\end{proof}

Before working towards a criterion for vanishing in degree 1, we point out that the previous lemma gives an easy criterion for vanishing in degree 0.

\begin{cor}\label{Vanishing in degree 0 for pi_p}
    Suppose that there exists some $\pi_p$-bounded element $\xi \in G$. Then:
    \begin{equation*}
        H_\mathrm{ct}^{0}(G, \pi_p ) = L^p(G, V)^{G} = \{0 \}.
    \end{equation*}
\end{cor}

\begin{proof}
    Every $f \in L^p(G, V)^{G}$ is $\pi(\xi)$-invariant, hence $f = 0$ by Lemma \ref{Pi-bounded-invariance}. 
\end{proof}

\subsection{Mautner's phenomenon and vanishing criterion}

Our next goal will be to give a criterion for vanishing in degree 1 for the representation $\pi_p$ using $\pi_p$-contracting elements. Whenever such an element exists, the following proposition allows us, for any cohomology class, to choose a representative that vanishes at the given contracting element.

\begin{prop}\label{invarianceByBCElements} \cite[2.1]{cornulier-tessera}
Let $\xi \in G$ be $\pi_p$-contracting and $b \in Z^1(G, \pi_p)$. There exists an element $c \in Z^1(G, \pi_p)$ such that $b -c \in B^1(G, \pi_p)$ and $c(\xi) = 0$. 
\end{prop}

\begin{proof}
We first claim that the sequence $(b(\xi^n))_{n >0}$ is Cauchy in $L^p(G,V)$. To see this, first notice that this sequence is bounded, as the cocycle relation yields:
\begin{equation*}
    b(\xi^n) = \sum_{i=0}^{n-1} \pi(\xi^i) b(\xi)
\end{equation*}
and hence:
\begin{equation*}
    ||b(\xi^n)|| \leq \big(\sum_{i=0}^{n-1} |||\pi(\xi^i)||| \big) ||b(\xi)|| \leq C ||b(\xi)||
\end{equation*}
where $C = \sum_{i=0}^{\infty} |||\pi(\xi^i)||| < \infty$ converges as $\xi$ is $\pi_p$-contracting and the operator norm is submultiplicative. Now let $n \geq m \geq 0$. From the cocycle relation we have:
\begin{equation*}
    b(\xi^n) - b(\xi^m) = \pi(\xi^m) b(\xi^{n-m})
\end{equation*}
and hence:
\begin{equation*}
     ||b(\xi^n) - b(\xi^m)|| \leq ||| \pi(\xi^m)||| \,  ||b(\xi^{n-m})|| \leq C  ||| \pi(\xi^m)||| \, ||b(\xi)||.
\end{equation*}
Hence the sequence $(b(\xi^n))_{n >0}$ is Cauchy in $L^p(G,V)$ and converges to a limit function $f \in L^p(G,V)$. The formula $c(g) = b(g) - f + \pi(g) f$ for $g\in G$ defines a cocycle $c$ such that $b -c \in B^1(G, \pi_p)$. We now show that $c(\xi) = 0$. We first see that:
\begin{equation*}
    ||c(\xi^n)|| \leq ||b(\xi^n) - f || + ||\pi(\xi^n) f||
\end{equation*}
and hence $||c(\xi^n)|| \to 0$ as $n \to + \infty$. From the cocycle relation we have:
\begin{equation*}
    c(\xi^n) = c(\xi) + \pi(\xi) c(\xi^{n-1}).
\end{equation*}
Since $||c(\xi^n)|| \to 0$ and $||\pi(\xi) c(\xi^{n-1})|| \to 0$ we have that $c(\xi) = 0$.
\end{proof}

The next step is propagating the vanishing created by this proposition. This involves some non-isometric variant of Mautner's phenomenon (recall that the classical Mautner's lemma concerns unitary, hence isometric, representations \cite[II. 3.2]{margulis-book}). We formulate this variant in terms of the following dynamical interpretation of the Levi decomposition \cite[2.2]{prasad}.

\begin{defn}
Let $\xi \in G$. We define:
\begin{align*}
    U_\xi & = \{ h \in G  ,   \xi^{-n} h \xi^n  \xrightarrow{n \to  + \infty} 1_{G} \}, \\
    P_\xi & = \{ h \in G  , \textrm{ the sequence } (\xi^{-n} h \xi^n )_{n>0} \textrm{ is bounded} \}, \\
     Z_\xi & = \{ h \in G  , \xi^{-1} h \xi = h \}.
\end{align*}
\end{defn}

The sets $P_\xi$, $U_\xi$ and $Z_\xi$ are subgroups of $G$. The subgroup $P_\xi$ contains both $U_\xi$ and $Z_\xi$ and these satisfy $U_\xi \cap Z_\xi = \{ 1_G\}$. When $G$ is a semisimple group and $\xi$ an element of a maximal split torus of $G$, we have that $P_\xi$ is a parabolic subgroup of $G$ and $U_\xi$ is its unipotent radical \cite[2.4 (i)]{prasad}, hence the Levi decomposition of $P_\xi$ may be written as $P_\xi= Z_\xi \ltimes U_\xi$.

\begin{prop}(Mautner's phenomenon)\label{Mautner-analogue}
Let $\xi \in G$ and let $b \in Z^1(G, \pi_p)$ be such that $b(\xi) = 0$. \\
1. (Classical Mautner's lemma) If $\xi$ is $\pi_p$-bounded, then $b(h) = 0$ for all $h \in U_\xi$. \\
2. (Contracting version) If $\xi$ is $\pi_p$-contracting, then $b(h) = 0$ for all $h \in P_\xi$. \\
3. (Commuting version) If $\xi$ is $\pi_p$-bounded, then $b(h) = 0$ for all $h \in Z_\xi$.
\end{prop}

\begin{proof}
Take $h \in G$. The cocycle relation implies:
\begin{equation*}
    b(\xi^{-1} h  \xi) = b(\xi^{-1} h ) = \pi(\xi^{-1}) b(h). 
\end{equation*}
Hence for all $n>0$ we obtain:
\begin{equation*}
    ||b(h)||_p \leq ||| \pi(\xi^n)|||_{L^p(G, V)} ||b(\xi^{-n} h \xi^n) ||_p.
\end{equation*}
1. Suppose that $h \in U_{\xi}$. Since $\xi$ is $\pi_p$-bounded, there exists $C>0$ such that $||b(h)||_p \leq C ||b(\xi^{-n} h \xi^n) ||_p$. Since $\xi^{-n} h \xi^n  \xrightarrow{n \to  + \infty} 1_{G}$, we have $||b(\xi^{-n} h \xi^n) ||_p  \xrightarrow{n \to  + \infty} 0$. Thus $b(h) = 0$.

2. Suppose that $h \in P_{\xi}$. The sequence $(\xi^{-n} h \xi^n)_n$ is bounded, so by continuity of $g \mapsto b(g) $ the term $||b(\xi^{-n} h \xi^n) ||_p$ is bounded. Since $\xi$ is $\pi_p$-contracting,  we have $\lim_{n \to \infty} ||| \pi(\xi^n)|||_{L^p(G, V)} = 0$ and hence $b(h)=0$.

3. Suppose that $h \in Z_{\xi}$. Hence the function $b(h)\in L^p(G, V)$ is $\pi(\xi)$-invariant. Since $\xi$ is $\pi_p$-bounded, Lemma \ref{Pi-bounded-invariance} gives $b(h) = 0$.
\end{proof}

\begin{rem}
    1. In what follows, we will use the contracting version most of the time. For groups of lower rank, we may not always dispose of enough contracting elements. In this case the commuting version can be very practical.

    2. For semisimple groups, we have $P_\xi = Z_\xi \ltimes U_\xi$. Hence in this case point 2 holds even when $\xi$ is only a $\pi_p$-bounded element, thanks to points 1 and 3.
\end{rem}

We come back to the setting of Corollary \ref{Hk(G,Lp(G))=Hk-1(R,Lp(R,H1(AN,Lp(AN))))}: let $G$ be a semisimple group, $\g$ a simple root of its restricted root system, $P_{\S \setminus \{ \g \} } = M^\g \ltimes H_\g$ the Levi decomposition of its corresponding maximal parabolic subgroup, $M^\g = K^\g A^\g N^\g $ be the Iwasawa decomposition of $M^\g $ and let $R^\g = A^\g \ltimes N^\g$. Recall: 
\begin{equation*}
    V_p : = H_\mathrm{ct}^1(H_\g, L^p(H_\g)).
\end{equation*} 
We apply the results of this section to the group $R^\g$ acting on $V_p$ via $\pi_0$ and on $L^p(R^\g, V_p)$ via $\pi$, where $\pi_0$ and $\pi$ are given by Corollary \ref{Hk(G,Lp(G))=Hk-1(R,Lp(R,H1(AN,Lp(AN))))}.

\begin{thm}\label{VanishingCriterion}
Suppose that there exists some $\pi_p$-contracting $\xi \in A^\g$ and that for each simple root $\s \in \S \setminus \{ \g \}$ there exists some $\pi_p$-bounded $\xi_\s \in A^\g$ such that $U_{\xi_\s} \cup Z_{\xi_\s}$ contains the root group $U_\s$. Then:
 \begin{equation*}
     H_\mathrm{ct}^{1}(R^\g, L^p(R^\g, V_p)) = \{0 \}.
 \end{equation*}
\end{thm}

\begin{proof}
Let $b \in Z^1(R^\g, \pi_p)$.
By Proposition \ref{invarianceByBCElements}, the cocycle $b$ can be chosen (without changing its cohomology class) so that $b(\xi) = 0$. Since the group $A^\g$ is abelian, the group $P_{\xi}$ contains $A^\g$ and thus by Proposition \ref{Mautner-analogue} point 2, we have $b(h) = 0$ for $h \in  A^\g$. This implies that $b(\xi_\s) = 0$ for all $\s \in \S \setminus \{ \g \}$. By Proposition \ref{Mautner-analogue} points 1 and 3, we have that $b(h) = 0$ for all $ h$ the sets $(U_{\xi_\s} \cup Z_{\xi_\s}) \cap R^\g$, which contain the root subgroups $U_{\s}$ for all $\s \in \S \setminus \{ \g \}$. The groups $U_{\s}$ for $\s \in \S \setminus \{ \g \}$ generate the group $N^\g$ \cite[21.9 (ii)]{borel-algebraic-groups}, therefore $b(h) = 0$ for all $h \in N^\g$. Thus $b = 0$. We conclude that $H_\mathrm{ct}^{1}(R^\g, L^p(R^\g, V_p)) = \{0 \}$ for every $p>1$.
\end{proof}

\subsection{Vanishing for semisimple, non-simple groups of rank $\geq 3$}

In this section we apply our vanishing criterion to semisimple, non-simple groups. The main point is that commutation makes some operator norms of $\pi_0$ to be equal to 1, hence we may reason directly as in \cite{cornulier-tessera}.

\begin{thm}\label{Vanishing for semisimple non-simple  groups}
    Let $G$ be a semisimple, non-simple group of rank $r \geq 3$ over a local field $F$. Then there exists a simple root $\g \in \S$ such that for all $p > \max \{  \Cdim(\partial H_\g), 1 \}$ we have:
    \begin{equation*}
        H_\mathrm{ct}^{2}(G, L^p(G)) = \{0 \}.
    \end{equation*}
\end{thm}

\begin{proof}

    Let $G_1, \ldots,  G_k$ be the simple factors of the group $G$. We split the proof in two cases. First suppose that there exists one factor $G_i$ of rank 1, associated to a simple root $\g \in \S$. The maximal parabolic subgroup $P_\g$ decomposes as $ P_\g = M^\g \ltimes H_\g$ and in this case $M^\g$ and $H_\g$ live in different factors, hence $M^\g$ and $H_\g$ commute, so $ P_\g = M^\g \times H_\g$. In particular, the action by conjugation by some element $g \in R^{\g}$ on $H_\g$ is trivial, and hence Proposition \ref{generalGrowthEstimate} gives:
    \begin{equation*}
        |||\pi(g)|||_{L^p(R^\g, V_p)} = \Delta_{R^\g}(g)^{-1}.
    \end{equation*}
    Since $r \geq 3$, the rank of $M^\g$, is at least 2 which means that the split torus $A^{\g}$ is of dimension at least 2. It is thus enough to consider two coweights $\xi_1, \xi_2 \in A^{\g}$ associated to two distinct simple roots $\s_1, \s_2 \in \S \setminus \{ \g\}$ satisfying $ \s_i(\xi_i) < 1$. In this way $\Delta_{R^\g}(\xi_i)^{-1} < 1$ and hence $\xi_1, \xi_2$ are $\pi_p$-contracting for all $p> 1$. The subgroup $P_{\xi_1}$ contains all the root groups $U_\s$ for $\s \in \S \setminus \{\g, \s_1\}$ and $P_{\xi_2}$ contains all the root groups $U_\s$ for $\s \in \S \setminus \{\g, \s_2\}$. Hence the conditions of Theorem \ref{VanishingCriterion} are satisfied and we showed:
    \begin{equation*}
         H_\mathrm{ct}^{1}(R^\g, L^p(R^\g, V_p)) = \{0 \}
    \end{equation*}
    for all $p>1$. Then the spectral sequence in Corollary \ref{Hk(G,Lp(G))=Hk-1(R,Lp(R,H1(AN,Lp(AN))))} gives the desired vanishing for $p > \Cdim(\partial H_\g)$.

    Suppose now that all factors have rank $\geq 2$. We may pick any simple root $\g$ in the root system of $G_1$. The parabolic subgroup $P_\g$ decomposes as $P_\g = M^{\g} \ltimes H_\g$, where the Levi factor $M^{\g}$ decomposes again in simple factors $M_1, \ldots, M_k$, with $M_i \subseteq G_i$ for all $i$. The main point is that the factor $M_2$ has rank $\geq 2$ and commutes with $H_\g$. Our hypothesis implies that there are at least two simple roots $\s_1, \s_2$ in the root system of $G_2$. Hence we may consider again coweights associated to these two simple roots  and reason as before to conclude that $H_\mathrm{ct}^{1}(R^\g, L^p(R^\g, V_p)) = \{0 \}$ for all $p> 1$. Again the spectral sequence gives the desired vanishing for $p > \Cdim(\partial H_\g)$. This second part of the proof is independent of the choice of root $\g$ and hence we can choose $\g$ in order to minimize $\Cdim(\partial H_\g)$.
\end{proof}

\section{Growth estimates and Heintze groups}

In Theorem \ref{VanishingCriterion} we proved that under the presence of enough contractions, we can show vanishing of cohomology. One would like to have a criterion to guarantee the existence of contracting elements for the $R^\g$-module $L^p(R^\g, H_\mathrm{ct}^1(H_\g, L^p(H_\g)))$ appearing in Corollary \ref{Hk(G,Lp(G))=Hk-1(R,Lp(R,H1(AN,Lp(AN))))}. Thanks to Proposition \ref{generalGrowthEstimate}, the only mysterious quantity that remains to study is the operator norm of the action $\pi_0$ of $R^\g$ on the space $H_\mathrm{ct}^1(H_\g, L^p(H_\g))$. In this section we find an upper bound of this norm that can be computed explicitly in combinatorial terms.

\subsection{Amenable hyperbolic groups with homotheties}

To find an upper bound on the operator norms of the action $\pi_0$ of $R^\g$ on the space $H_\mathrm{ct}^1(H_\g, L^p(H_\g))$, we will first obtain some preliminary results in the more general setting of amenable hyperbolic groups and contracting automorphisms developed in \cite{caprace-cornulier-tessera-monod}, with the supplementary condition that we will contract using a homothety.

Let $U$ be a locally compact second countable group. By \cite{struble}, the group $U$ admits a proper left-invariant metric $d$ metrizing its topology. Suppose moreover that $(U,d)$ has a (non-trivial) \textit{homothety} $\a \in \mathrm{Aut}(U)$, that is, a group automorphism such that for some $\lambda > 1$ we have:
\begin{equation*}
    d(\a(x), \a(y)) = \lambda^{-1} d(x,y).
\end{equation*}
By \cite[6.5]{caprace-cornulier-tessera-monod}, the group $U$ is nilpotent and unimodular. Denote by $\Omega$ the compact unit ball of $(U,d)$ centered at $e_U$ and write $|| x || = d_U(x,e)$. Notice that $\bigcap_{k \geq 0} \a^k(\Omega) = \{ e_{U} \}$ and $\bigcup_{k \leq 0} \a^k(\Omega) = U$.

The set $S = \{\a^{\pm 1}\} \times \Omega $ generates the group $H = \langle \a \rangle \ltimes U$. The group $H$ is amenable and Gromov-hyperbolic when endowed with the word metric $|\cdot |_S$ with respect to $S$. Denote by $B_S(n)$ the closed ball of radius $n$ centered at $e_H$.  

The following lemmata come from \cite{caprace-cornulier-tessera-monod} and allow us to estimate distortion of $(U,d)$ inside $(H, | \cdot | _S)$. Notice in particular that only the upper bound needs to have a homothety.

\begin{lem}\cite[Lemma 4.7]{caprace-cornulier-tessera-monod} For $m \geq 0$, we have:
\begin{equation*}
    \a^{-m}(\Omega) \subseteq B_S(2m +1) \cap U.
\end{equation*}
\end{lem}

\begin{lem}\label{geodesics in Heinzte group} \cite[Lemma 4.8]{caprace-cornulier-tessera-monod}
There exists $C > 0$ such that for all $x \in U$, there exist $i, j \in \na$ with $j \leq C$ and $x_1, \ldots, x_j \in \Omega$ such that $x = \a^{-i} x_1 \ldots x_j  \a^i$ and $2i \leq |x |_S + C$.
\end{lem}

\begin{prop}\label{bounds for distortion}
There exists $C > 1$ such that for all $x \in U$:
\begin{equation*}
 C^{-1} \l^{\frac{1}{2} |x |_S } \leq  || x || \leq C\l^{\frac{1}{2} |x |_S }.
\end{equation*}
\end{prop}

\begin{proof}
Let $m$ be such that $x \in \a^{-m}(\Omega) \setminus \a^{-m +1}(\Omega)$. Hence $|x|_S \leq 2m + 1$ and:
\begin{equation*}
    || x || \geq \l^{m-1} \geq \l^{\frac{1}{2} (|x |_S - 3) }.
\end{equation*}
 Write $x = \a^{-i} x_1 \ldots x_j  \a^i$ as a word in the alphabet $S$ given by Lemma \ref{geodesics in Heinzte group}. Then as an element of $U$ we have $x = \a^{-i}(x_1 \ldots x_j )$ and so:
 \begin{equation*}
     || x || = \l^i || x_1 \ldots x_j || \lesssim \l^i \lesssim \l^{\frac{1}{2} |x |_S }.
 \end{equation*}
\end{proof}

\begin{lem}\label{modular_function_of_homothety}
Let $Q$ denote the Hausdorff dimension of the metric space $(U,d)$. \\
$1.$ The Haar measure on $U$ is equivalent to the $Q$-Hausdorff measure of $(U,d)$. \\
$2.$ Let $\D_H$ denote the modular function of $H$. Then we have: $\D_H(\a) = \l^Q$.
\end{lem}

\begin{proof}
$1.$ Fix some left-invariant Haar measure $\mu$ on $U$. Then the measures $\a_*^k \mu$ on $U$ are also left-invariant for $k \in \ze$. Such measures are unique up to positive scalar, so there exists some $s \in \re_+^*$ such that $\a_*^k \mu = s^k \mu$ for all $ k\in \ze$.
We now show that $\mu$ is an Ahlfors $Q$-regular measure. For this, let $r > 0$ and $k \in \ze$ so that $\l^k \leq r < \l^{k+1}$. We have:
\begin{equation*}
    \mu(B(\l^k)) \leq \mu(B(r)) < \mu(B(\l^{k+1})).
\end{equation*}
Since $\a$ is a homothety, we have $B(\l^k) = \a^{-k}\Omega$ and thus 
\begin{equation*}
    \mu(B(\l^k)) = \a_*^k\mu(\Omega) = s^k \mu(\Omega).
\end{equation*}
We have $\mu(\Omega) < \infty$ since $\mu$ is a Radon measure (and $\Omega$ is compact, because $d$ is a proper metric) and $\mu(\Omega) > 0$ because $\Omega$ contains an open set as $d$ metrizes the topology of $U$. From this we deduce that there exists $C>0$ such that:
\begin{equation*}
    C^{-1} r^{\frac{\log s}{\log \l}} \leq \mu(B(r)) \leq C r^{\frac{\log s}{\log \l}}.
\end{equation*}
This implies that $\mu$ is Ahlfors $(\frac{\log s}{\log \l})$-regular and that the Hausdorff dimension $Q$ of the metric space $(U,d)$ must satisfy  $Q = \frac{\log s}{\log \l}$, that is $s = \l^Q$.

$2.$ Since the groups $U$ and $\langle \a \rangle \,\cong \, \ze$ are unimodular, the modular function of the semi-direct product $H = \langle \a \rangle \ltimes U $ is computed by the following formula. For any measurable set $E$ in $U$ we have:
\begin{equation*}
    \mu(\a^{-1}(E)) = \Delta_H(\a) \mu(E).
\end{equation*}
For $E = \Omega$ we have $0 < \mu(\Omega) < \infty$ and: 
\begin{equation*}
\Delta_H(\a) \mu(\Omega) = \mu(\a^{-1}(\Omega)) = \a_*\mu(\Omega) = \l^Q \mu(\Omega),
\end{equation*}
since in the proof of 1. we showed that $s = \l^Q$. Hence $\Delta_H(\a) = \l^Q$.
\end{proof}

\subsection{Growth estimates for Heintze groups}

We can now obtain an upper bound on the operator norms of the action $\pi_0$ of $R^\g$ on the space $H_\mathrm{ct}^1(H_\g, L^p(H_\g))$. In fact we will only do it for the restricted action of the torus $A^\g$, but this will be enough for our purposes. We come back to the usual Lie theoretic setting and use the usual notation. Our first task is to show that we can apply the results from the previous subsection. Fix $s \in S_\g$ such that $\l : = |\g(s^{-1})| > 1$ for the rest of the section.

\begin{prop}
    There exists a proper, left-invariant metric $d_\g$ on $U_\g$ so that conjugation by $s$ is a homothety of ratio $\l^{-1}$ on the metric space $(U_\g, d_\g)$.
\end{prop}

\begin{proof}
In the real case, $U_\g$ is a Carnot group for the gradation on its Lie algebra $\mathrm{Lie}(U_\g) = \bigoplus_{\a \in \Phi^+, n_\g(\a) > 0} \ag_\a$ given by
\begin{equation*}
    \ag_k := \bigoplus_{\a \in \Phi^+, n_\g(\a)= k} \ag_\a
\end{equation*}
for $k>0$. This means that $\mathrm{Lie}(U_\g)$ is generated by the subspace $\ag_1$. This is shown for instance in \cite{yamaguchi}, see there Lemma 3.8 and the discussion at the end of Section 3.4. Since $U_\g$ is a Carnot group, it carries a \textit{Carnot-Carathéodory metric} \cite[p. 3]{pansu89-carnot}, which satisfies the conditions of the Proposition \cite[p. 4]{pansu89-carnot}. 

Suppose now that the local field $F$ is non-Archimedean, consider its ring of integers $\mathcal{O}$ and let $\Omega$ be the compact-open subgroup $\prod_{\a \in \Phi_{n_\g > 0}} \theta_\a(\mathcal{O})$. 
Notice that $\bigcap_{k \geq 0} s^k.\Omega = \{ e_{U_\g} \}$ and $\bigcup_{k \leq 0} s^k.\Omega = U_\g$.
These two conditions allow us to define a valuation $v(x) = \sup \{ k \in \ze, x \in s^k.\Omega \}$ for $x \in U_\g$. Since $\Omega$ is a subgroup, it satisfies the inequality $v(xy) \geq \min \{ v(x), v(y) \}$ for all $x, y \in U_\g$. We may define the norm 
\begin{equation*}
    ||x||_\g := \l ^{- v(x)},
\end{equation*}
on $U_\g$ and the distance $d_\g (x, y) := ||x^{-1} y ||_\g$ for all $x,y \in U_\g$. The distance $d_\g$ on $U_\g$ satisfies the ultrametric inequality, is left-invariant and satisfies:
\begin{equation*}
    d_\g(s.x, s.y) = \l^{-1} d_\g (x, y)
\end{equation*}
for all $x, y  \in U_\g$.
\end{proof}

\begin{rem}
1. The construction in the non-Archimedean case also works for a group $U$ with compact neutral component $U^0$ and with a compacting automorphism $\a$, meaning that there exists some compact subset $\Omega$ such that for any compact subset $K \subseteq U$, there exists $k_0$ such that for all $k \geq k_0$, we have $\a^k(K) \subseteq \Omega$. By \cite[4.5.ii]{cornulier-tessera}, we may choose $\Omega$ to be a compact-open subgroup such that $\a^k(\Omega) \subseteq \Omega$ for all $k > 0$. If $ L = \bigcap_{k \geq 0}\a ^k(\Omega)$, then our proof defines a $U$-invariant distance on $U/L$ that satisfies the ultrametric inequality and for which the automorphism induced by $\a$ on $U/L$ is a homothety.

2. We can also consider this construction when $U_\g$ is a real Lie group by replacing $\Omega$ with a product of compact intervals containing $0$. A minor problem arises from the fact that in this case $\Omega$ is not a subgroup, so the corresponding valuation will only satisfy $v(xy) \geq \min \{ v(x), v(y) \} - C$ for some constant $C> 0$. Thus the formula $||x||_\g = e^{- v(x)}$ defines only a left-invariant quasi-metric $d_\g$. This can be salvaged by considering a power $d_{\g}^{a}$ of $d_\g$ for $0<a<1$ sufficiently small. The main problem is that we do not have a good control of the Hausdorff dimension of this metric (it can be really big), as opposed to the Carnot-Carathéodory metric, which has minimal dimension in the conformal gauge of $\partial H_\g$.
\end{rem}

Suppose again that $F$ is a (real or non-Archimedean) local field. We will consider a function that is Lipschitz-equivalent to the metric $d_\g$ and easier for computations. To construct it, we first identify $S_\g$-equivariantly each factor $U_\a$ to $\ag_\a$ endowed with some vector space norm $| \cdot |_\a$ (in the non-Archimedean case, we choose these norms to be compatible with the absolute value $| \cdot |$ on $F$ such that $\Hdim(F, | \cdot |) = 1$ in the sense that $|x v|_\a = |x| |v|_\a$ for all $x \in F$ and $v \in \ag_\a$). We decompose an element $x \in U_\g$ in coordinates $x = (x_\a)_{\a \in \Phi_{n_\g > 0}}$, where $x_\a \in U_\a$, and define:
\begin{equation*}
    N_\g(x) :=  \max\{| x_\a|^{1 / n_\g(\a)}, \a \in \Phi_{n_\g > 0} \}.
\end{equation*}
In this way $N_\g(s.x) =  \max\{|\a(s) x_\a|^{1 / n_\g(\a)}, \a \in \Phi_{n_\g > 0} \} = \l^{-1} N_\g(x)$ for all $x \in U_\g$. Indeed, as $s \in S_\g = \bigcap_{\s \in \S \setminus \{ \g \} } \ker \s$, we have $|\a(s)| = |\g(s)|^{n_\g(\a)}$ and $|\a(s)|^{1 / n_\g(\a)} = |\g(s)|$ for all $\a \in \Phi_{n_\g > 0}$. As the function $N_\g$ is positive on $U_\g \setminus \{ 1 \}$ and $s$ acts on it as a homothety of ratio $\l^{-1}$, the left-invariant function $N_\g(x, y) = N_\g(x^{-1}y)$ is Lipschitz-equivalent to the metric $d_\g$. 

Using the function $N_\g$ we may compute the Hausdorff dimension $Q_\g$ of the metric space $(U_\g, d_\g)$. Indeed, we have:
\begin{equation*}
    Q_\g = \sum_{\a \in \Phi_{n_\g > 0}} \Hdim(\ag_\a , | \cdot |_\a^{1/n_\g(\a)}) = \sum_{\a \in \Phi_{n_\g > 0}} n_\g(\a) m(\a).
\end{equation*}

For $g \in A^\g$ define:
 \begin{equation*}
     ||g||_\g := \max_{\a \in \Phi_{n_\g > 0}}(|\a(g)|^{1/n_\g(\a)} ).
 \end{equation*}

\begin{lem}\label{torus_acting_on_U}
There exists $C> 0$ such that for all $g \in A^\g$ and all $x, y \in U_\g$ we have:
    \begin{equation*}
    d_\g (g.x, g.y) \leq C ||g||_\g d_\g(x,y)
\end{equation*}
\end{lem}
 
\begin{proof}
Directly from the definitions of $N_\g$ and $||g||_\g$ it follows that for all $g \in A^{\g}$ and $x \in U_\g$ we have $N_\g(g . x) \leq ||g||_\g N_\g(x)$. Since $N_\g$ and $d_\g$ are Lipschitz-equivalent, there exists $c > 0$ such that $(1/c) d_\g \leq N_\g \leq c d_\g$ and hence
\begin{equation*}
    d_\g (g.x, g.y) \leq c N_\g(g.x, g.y) \leq c ||g||_\g N_\g(x, y) \leq c^2 ||g||_\g d_\g(x,y).
\end{equation*}
\end{proof}  

The main result of the section is the following:

\begin{thm}\label{CohomologyGrowthThm}
There exists a constant $C = C(p, \g) > 0$ such that for all $g \in A^\g$:
\begin{equation*}
    |||\pi_0(g)|||_{V_p}^p \leq C ||g^{-1}||_\g^{Q_\g}.
\end{equation*}
\end{thm}

\begin{rem}
This inequality is optimal (up to the constant $C>1$) in the case of $G = \SL(r+1, \re)$ with simple root $\g = e_1 - e_2$. Indeed, we have $M_{\S \setminus \{\g\}} =  \SL(r, \re)$ acting on $U_\g = \re^r$ by the natural linear action and $Q_\g = r$. The space $V_p$ may be identified (as a continuous $M_{\S \setminus \{\g\}}$-module) with the Besov space $B_{p,p}^{r/p}(\re^r)$, that is, the Banach space of measurable functions $u : \re^r \to \re$ such that 
\begin{equation*}
    ||u||_{B_p}^p = \int_{\re^r \times \re^r} \frac{|u(x) - u(x')|^p}{d(x, x')^{2r}} \mathrm{d}x \mathrm{d}x' < \infty,
\end{equation*}
modulo constant functions \cite[5.2]{pansu89}.
For $g = \mathrm{diag}(2^{1-r}, 2, \ldots, 2) \in M_{\S \setminus \{\g\}}$, we construct the function $u \in B_{p,p}^{r/p}(\re^r)$ such that $u(x) = 1$ on the cube $Q_1 = [1,2]^r$ and $u(x) = 0$ on the cube $Q_2 = [-2, -1] \times [1,2]^{r-1}$ disjoint from $Q_1$ and restricting the previous integral to $Q_1 \times Q_2$ we may see that $||\pi_0(g^k) u ||^p_{B_p} \gtrsim ||g^{-k}||^r ||u||^p_{B_p}$.
\end{rem}

After combining the inequality in Theorem \ref{CohomologyGrowthThm} with Proposition \ref{generalGrowthEstimate}, we obtain the following corollary.

\begin{cor}\label{CohomologyGrowthCor}
There exists a constant $C = C(p, \g) > 0$ such that for all $g \in A^\g$:
\begin{equation*}
    |||\pi_p(g)|||_{V_p}^p \leq C \Delta_{R^\g}(g)^{-1} ||g^{-1}||_\g^{Q_\g}.
\end{equation*}
In particular: \\
any $g \in A^\g$ such that $\Delta_{R^\g}(g)^{-1} ||g^{-1}||_\g^{Q_\g} < 1$ is $\pi_p$-contracting for every $p>1$, \\
any $g \in A^\g$ such that $\Delta_{R^\g}(g)^{-1} ||g^{-1}||_\g^{Q_\g} \leq 1$ is $\pi_p$-bounded for every $p>1$.
\end{cor}

In order to prove Theorem \ref{CohomologyGrowthThm} we apply results from the previous section to the group $U = U_\g$ with its left-invariant metric $d = d_\g$ and the automorphism $\a$ is conjugation by some non-trivial $s \in S_\g$ such that $ \l^{-1} := |\g(s) | < 1$. The main idea of the proof is to use the cocycle relation at each step of a word as in Lemma \ref{geodesics in Heinzte group} and compute the contribution of each step using Lemma \ref{modular_function_of_homothety}.

\begin{proof}[Proof of theorem \ref{CohomologyGrowthThm}]
Here we replace $H_\g$ with its cocompact subgroup $ \langle s \rangle \ltimes U_\g$. Let $\rho_{H_\g}$ be the right regular representation of $H_\g$ on $L^p(H_\g)$. As $H_\g$ is compactly generated, with compact generating set $S = \{s^{\pm 1}\} \times \Omega$, we may consider $Z^1(H_\g, L^p(H_\g))$ as a Banach space with norm $|| b ||_{S, p} = \sup_{x \in S} ||b(x) ||_p$. Let $b \in Z^1(H_\g, L^p(H_\g))$. Since the action by conjugation of $R^\g$ on $H_\g$ preserves the volume, we have:
\begin{equation*}
    ||\pi_0(g) b ||_{S,p} = \sup_{x \in S} ||b(g^{-1} x g) ||_p.
\end{equation*}

Fix some $x \in S$. We suppose that $x \in \{ s \} \times \Omega$, the situation being similar for $x \in \{ s^{-1}\} \times \Omega$. Write $x = y s$, with $y \in \Omega$. Lemma \ref{geodesics in Heinzte group} says that there exist $C>0$ and $n, j \in \na$ with $j \leq C$ and $x_1, \ldots, x_j \in \Omega$ such that 
\begin{equation*}
    g^{-1}yg = s^{-n} x_1 \ldots x_j  s^{n},
\end{equation*}
and $2n \leq |g^{-1}yg |_S + C $.

Let $s_i \in S$ be the $i$-th term (in the alphabet $S$), from left to right, of the word 
\begin{equation*}
    g^{-1}xg =  s^{-n} x_1 \ldots x_j  s^{n+1}.
\end{equation*} 
More precisely, 
\begin{equation*}
s_i = 
\begin{cases}
  s^{-1} & \text{ for } 1 \leq i\leq n,\\ 
 x_{i-n} & \text{ for } n+1 \leq i \leq n+j, \\
 s & \text{ for } n+j+1 \leq i \leq 2n + j +1.
\end{cases}
\end{equation*}

\noindent Let $h_i = s_1 \ldots s_i$ for all $i>0$ and $h_0 = 1$. Iterating the cocycle relation for $b$ yields:
\begin{equation*}
    b(g^{-1}xg) = \sum_{i= 1}^{2n+j+1} \rho_{H_\g}(h_{i-1})b(s_i).
\end{equation*}
From here we obtain:
\begin{equation*}
      ||b(g^{-1}xg)||_p \leq \sum_{i= 1}^{2n+j+1} ||\rho_{H_\g}(h_{i-1})b(s_i)||_p = \sum_{i= 1}^{2n+j+1} \D_{H_\g}(h_{i-1})^{-1/p}||b(s_i)||_p.
\end{equation*}
Using Lemma \ref{modular_function_of_homothety}, we have that:
\begin{equation*}
\D_{H_\g}(s_{i})^{-1} = 
\begin{cases}
  \l^{Q_\g} & \text{ for } 1 \leq i\leq n,\\ 
 1 & \text{ for } n+1 \leq i \leq n+j, \\
 \l^{-Q_\g} & \text{ for } n+j+1 \leq i \leq 2n + j +1.
\end{cases}
\end{equation*}

Thus:
\begin{equation*}
\D_{H_\g}(h_{i})^{-1} = 
\begin{cases}
  \l^{i Q_\g} & \text{ for } 1 \leq i\leq n,\\ 
 1 & \text{ for } n+1 \leq i \leq n+j, \\
 \l^{(2n + j + 1 - i)Q_\g} & \text{ for } n+j+1 \leq i \leq 2n + j +1.
\end{cases}
\end{equation*}

Hence:
\begin{align*}
    ||b(g^{-1}xg)||_p & \leq \sum_{k=0}^n \l^{ k Q_\g/p} ||b||_{S,p} + C \l^{ n Q_\g/p} ||b||_{S,p} + \sum_{k=0}^n \l^{ k Q_\g/p} ||b||_{S,p}  \\
    & \leq C_1  \l^{ n Q_\g/p} ||b||_{S,p}.
\end{align*}
where $C_1 =2  \frac{1}{1- \l^{-Q_\g/p}} + C$. Using $2n \leq |g^{-1}yg |_S + C $ and Proposition \ref{bounds for distortion} we obtain
\begin{equation*}
    \l^n \leq \l^{\frac{1}{2} |g^{-1}yg |_S + C } \leq C_2 || g^{-1}y g  ||_\g.
\end{equation*}

Since $g \in A^{\g}$, Lemma \ref{torus_acting_on_U} says that $|| g^{-1}y g  ||_\g \leq C_3 ||g^{-1}||_\g  ||y||_\g$. We have that $y \in \Omega$, so $||y||_\g \leq 1$ and $\l^n \leq C_4 ||g^{-1}||_\g$.
From this we may conclude that:
\begin{equation*}
    ||\pi_0(g) b ||_{S,p} = \sup_{x \in S} ||b(g^{-1} x g) ||_p \leq C_5 || g^{-1}  ||_\g^{Q_\g /p} ||b||_{S,p}.
\end{equation*}
Hence $|||\pi_0(g)|||_{Z^1(H_\g, L^p(H_\g))} \leq C_5  || g^{-1}  ||_\g^{Q_\g /p}$ and the operator induced on the quotient Banach space $V_p = H_\mathrm{ct}^1(H_\g, L^p(H_\g))$ has smaller operator norm:
\begin{equation*}
    |||\pi_0(g)|||_{V_p} \leq C_5  || g^{-1}  ||_\g^{Q_\g /p}.
\end{equation*}

\end{proof}

\section{Existence of contracting elements for simple groups}

In this section we want to show existence of $\pi_p$-contracting elements as in our vanishing criterion (Theorem \ref{VanishingCriterion}) via our estimate (Corollary \ref{CohomologyGrowthCor}) for simple groups appearing in the statement of Theorem \ref{Vanishing-in-degree-2-Intro}. We first translate the condition of Theorem \ref{VanishingCriterion} in combinatorial terms as an inequality depending only on a root system and the multiplicities of the roots. Then, using the classification of simple algebraic groups over local fields in its more classical form (in the sense that we give a full list of the groups in the classification), we obtain Theorem \ref{Vanishing-in-degree-2-Intro} thanks to a case-by-case verification of the inequality. We sum up these results in tables. Taking into account multiplicities explains why we cannot stand only by the root system and why we go back to classical presentations.

We conclude this section by obtaining Corollary \ref{Uniform-Vanishing-for-Admissible-Groups} as a byproduct of this combinatorial study and by comparing our results with those from \cite{bourdon-remy-vanishings}.

\subsection{Combinatorial reformulation}

Corollary \ref{CohomologyGrowthCor} says that to guarantee the existence of $\pi_p$-contracting (resp. $\pi_p$-bounded) elements, it is enough to compute the term $\D_{R^\g}(g)^{-1} ||g^{-1}||_\g^{Q_\g}$ for $g \in A^\g$ and see that this term is $< 1$ (resp. $\leq 1$). In practice, $\pi_p$-bounded elements that are not $\pi_p$-contracting appear only in some exceptional low rank cases.

Our next goal is to write the expression $\D_{R^\g}(g)^{-1} ||g^{-1}||_\g^{Q_\g}$ in terms of root systems. We may begin with $\D_{R^\g}(g)$. For $g \in A^{\g}$, the modular function of $R^\g$ can be expressed as:
\begin{equation*}
    \D_{R^\g}(g)^{-1} = \prod_{\a \in (\Phi_{\S \setminus \{ \g \}})^+} |\a(g)|^{m(\a)}.
\end{equation*}
Choose some simple root $\s \in \S, \s \neq \g$. Let $g \in A^{\g} \cap \bigcap_{\t \in \S\setminus \{ \g, \s \}} \ker \t$ such that $|\s(g)| > 1 $. This choice is made so that the subgroup $P_g$ of $R^\g$ is big. Indeed, $g$ commutes with $U_\t$ for all $\t \in \S\setminus \{ \g, \s \}$, therefore $P_g \supset U_\t$ for all $\t \in \S\setminus \{ \g, \s \}$. The modular function of $g$ may be expressed as follows:
\begin{equation*}
    \D_{R^\g}(g)^{-1} = \prod_{\a \in \Phi_{n_\s > 0} \cap (\Phi_{\S \setminus \{ \g \}})^+} |\s(g)|^{n_\s(\a) m(\a)}.
\end{equation*}
On the other hand, since $g \in \bigcap_{\t \in \S\setminus \{ \g, \s \}} \ker \t$, we have that:
\begin{equation*}
    |\a(g)| = |\g(g)|^{n_\g(\a)} |\s(g)|^{n_\s(\a)}
\end{equation*}
for all $\a \in \Phi^+$, so that:
\begin{equation*}
    ||g||_\g = |\g(g)| \max_{\a \in \Phi_{n_\g > 0}}\{|\s(g)|^{n_\s(\a)/n_\g(\a)}\}.
\end{equation*}
Thus the condition $\D_{R^\g}(g) ||g||_\g^{Q_\g} < 1$ is equivalent to:
\begin{equation*}
    \sum_{\a \in \Phi_{n_\s > 0} \cap (\Phi_{\S \setminus \{ \g \}})^+} n_\s(\a) m(\a) >  Q_\g  (\max_{\a \in \Phi_{n_\g > 0}}\{ \frac{n_\s(\a)}{n_\g(\a)} \} + \frac{\log |\g(g)|}{\log |\s(g)|}).
\end{equation*}
Since $g \in A^{\g}$, the terms $\log |\t(g)|$ for $\t \in \S$ satisfy a non trivial linear relation. Moreover, $g \in \bigcap_{\t \in \S\setminus \{ \g, \s \}} \ker \t$, so that $\log |\g(g)|$ depends linearly on $\log |\s(g)|$ and thus $\frac{\log |\g(g)|}{\log |\s(g)|} =: C_\s$ is a constant that depends only on the root system $\Phi$ and not on the choice of $g$. 
We proved the following criterion:
\begin{prop}\label{Combinatorial-Inequality}
Let $\s_1 , \s_2 \in \S \setminus \{ \g \}$ be two distinct simple roots satisfying: 
\begin{equation*}
        \sum_{\a \in \Phi_{n_{\s_i} > 0} \cap (\Phi_{\S \setminus \{ \g \}})^+} n_{\s_i}(\a) m(\a) \, \geq \, Q_\g \cdot (\max_{\a \in \Phi_{n_\g > 0}}\{ \frac{n_{\s_i}(\a)}{n_\g(\a)} \} + C_{\s_i})
\end{equation*} 
for $i =1,2$, with at least one of the two satisfying strict inequality. Then hypotheses of Theorem \ref{VanishingCriterion} are satisfied.
\end{prop}

\begin{rem}
    The left-hand side and the constant $Q_\g$ depend on the root system $\Phi$ and the multiplicities of the roots. On the other hand, the second factor of the right-hand side depends only on $\Phi$.
\end{rem}

\subsection{Vanishing in degree 2 for classical simple groups}
We start a case-by-case study of simple algebraic groups over a local field $F$ to find those that satisfy the conditions of Proposition \ref{Combinatorial-Inequality}. We restrict ourselves to the so-called classical groups, leaving the exceptional groups aside.

\paragraph{Classification of classical simple groups over local fields} 
Let $F$ be a local field. We recall some parts of the classification of (absolutely) simple algebraic groups over $F$ up to isogeny (see \cite{tits-boulder} for the classification in the more general algebraic group setting, \cite{helgason} for the real case and \cite{tits-corvallis} for the non-Archimedean case). We do not need the full power of the classification via Satake diagrams as we only use its phrasing in classical terms. This means that we look at the associated ordinary Dynkin diagrams and take into account multiplicities. The classification result is that, apart for groups of exceptional type, the only sources of families of simple algebraic groups over $F$ are the groups $\mathrm{SL}_n$ and, in the terminology of \cite[23.8]{borel-algebraic-groups}, groups preserving $(\epsilon , \s)$-hermitian forms. The latter split into symmetric, skew-symmetric, hermitian and skew-hermitian forms and allow to deal with them in a uniform setting. This classification result applies both to the real and the non-Archimedean setting, though the restrictions on the parameters of the forms are different in each case (in most cases, the non-Archimedean case is more restrictive, except for $\mathrm{SL}_n$ and the skew-hermitian case).

We now list the families of all (isogeny classes of) classical absolutely simple algebraic groups over $F$ of split rank $r$. Our list is based on \cite[4.4]{tits-corvallis}, though that list only concerns the non-Archimedean case. For the real groups we mention, we compare each situation with the list in \cite[X, p.532-534]{helgason}.

- The special linear group $\mathrm{SL}_{r+1} (D)$ over $D$, where $D$ is a $d^2$-dimensional central division $F$-algebra. In the real case, $d = 1 $ or $d = 2$, corresponding to $\mathrm{SL}_{r+1}(\re)$ and $\mathrm{SL}_{r+1}(\mathbb{H})$. In the non-Archimedean case there is no restriction on $d$.

- The special unitary group $\mathrm{SU}(h)$ of a hermitian form $h$ in $n$ variables and Witt index $r$ over a quadratic extension of $F$. In the real case there is no restriction on $n$ and $r$, giving the groups $\mathrm{SU}(r, n-r)$. In the non-Archimedean case, $2r \leq n \leq 2r +2$.

- The special orthogonal group $\mathrm{SO}(q)$ of a quadratic form $q$ in $n$ variables and Witt index $r$. In the real case there is no restriction on $n$ and $r$,  giving the groups $\mathrm{SO}_{r, n-r}(\re)$. In the non-Archimedean case, $2r \leq n \leq 2r +4$.

- The symplectic group $\mathrm{Sp}_{2r}(F)$.

- The special unitary group $\mathrm{SU}(\Tilde{h})$ of a quaternion hermitian form $\Tilde{h}$ in $n$ variables and Witt index $r$. In the real case there is no restriction on $n$ and $r$, giving the groups $\mathrm{Sp}_{2r, 2(n-r)}(\re)$. In the non-Archimedean case, the Witt index is always maximal, that is, $n= 2r$ or $ 2r +1$. 

- The special orthogonal group $\mathrm{SO}(\Tilde{q})$ of a quaternion skew-hermitian form $\Tilde{q}$ in $n$ variables and Witt index $r$. In the real case we have $n= 2r$ or $2r+1$, giving the groups $\mathrm{SO}^*(2n)$. In the non-Archimedean case, $2r \leq n \leq 2r +3$.

\paragraph{Statement of the vanishing theorem}
The result we will prove is the following:

\begin{thm}\label{Vanishing in degree 2}
Let $G$ be $\SL(4, D)$ or one of the simple groups listed above with split rank $r \geq 4$. If $G$ is not of type $D_4$, there exists a simple root $\g$ such that $H^2_{\mathrm{ct}}(G, L^p(G)) = \{0\}$ for all $p>\max\{1, \Cdim(\partial H_\g) \}$.
\end{thm}

In order to prove this theorem, we need to verify the conditions of Proposition \ref{Combinatorial-Inequality}. For this, we need to choose a simple root $\g$ to construct our parabolic subgroup and then two simple roots $\s_1, \s_2 \in \S\setminus \{ \g \}$ to construct two contractions. Each simple root $\g \in \S$ partitions the set of positive roots $\Phi^+$ into two disjoint sets, $\Phi_{n_\g > 0}$ and  $\Phi_{\S \setminus \{\g\}}$. The idea is to choose $\g$ so that $|\Phi_{n_\g > 0}|$ is as small as possible and $|\Phi_{\S \setminus \{\g\}}|$ is as big as possible (in the corresponding Dynkin diagram, this is done by choosing an extremal vertex).

\paragraph{Root systems}
We will start by computing the quantities in the inequality of Proposition \ref{Combinatorial-Inequality} that depend only on the root system $\Phi$. This is the constant:
\begin{equation*}
    D_{\s_i} = \max_{\a \in \Phi_{n_\g > 0}}\{ \frac{n_{\s_i}(\a)}{n_\g(\a)} \} + C_{\s_i}.
\end{equation*}

\begin{table}
\centering
\resizebox{\columnwidth}{!}{
\begin{tabular}{ |L|L|L|L|L|L| } 
\hline
 & A_r &  B_r & C_r & BC_r & D_r  \\
  \hline
\hline
 V = X(S) \otimes \re & \{ e_1 + e_2 + \ldots + e_{r+1} = 0 \} &  \re^r &  \re^r & \re^r &  \re^r \\ 
  \hline
 \textrm{Choice for }\g & \t_1 = e_1 - e_2 &  \t_1 = e_1 - e_2  &  \t_1 = e_1 - e_2  & \t_1  = e_1 - e_2 &  \t_1 = e_1 - e_2 \\ 
  \hline 
 n_\g(\a) \textrm{ for } \a \in \Phi_{n_\g > 0} & 1 & 1 & n_\g(2e_1) =2, \textrm{ else } 1 & n_\g(2e_1) =2, \textrm{ else } 1 & 1 \\ 
  \hline
   \textrm{Equation for } A^\g    & \sum_{i=1}^r (r+1-i)\t_i = 0 &\sum_{i=1}^r \t_i = 0 &\sum_{i=1}^{r-1} 2 \t_i + \t_r = 0  & \sum_{i=1}^r \t_i = 0 &\sum_{i=1}^{r-2} 2 \t_i + \t_{r-1} + \t_r = 0  \\ 
   \hline
   \hline
 \textrm{Choice for }\s_1    & \t_2 = e_2 - e_3  &\t_r = e_r & \t_r = 2e_r & \t_r = e_r &\t_r = e_{r-1} + e_r  \\ 
   \hline
   C_{\s_1}    & -(r-1)/r &-1 & -1/2 & -1 &-1/2 \\ 
   \hline
    \max_{\a \in \Phi_{n_\g > 0}}\{ \frac{n_{\s_1}(\a)}{n_\g(\a)} \}    & 1 &2 & 1 & 2 &1 \\ 
   \hline
      D_{\s_1}    & 1/r &1 & 1/2 & 1 & 1/2 \\ 
   \hline
   \hline
 \textrm{Choice for } \s_2& \t_3 = e_3 - e_4 &  \t_{r-1} = e_{r-1} - e_r  &  \t_{r-1} = e_{r-1} - e_r    & \t_{r-1} = e_{r-1} - e_r   & \t_{r-1} = e_{r-1} - e_r  \\ 
  \hline
 C_{\s_2}    & -(r-2)/r &-1 & -1 &  -1 &-1/2 \\ 
 \hline
    \max_{\a \in \Phi_{n_\g > 0}}\{ \frac{n_{\s_2}(\a)}{n_\g(\a)} \}    & 1 &2 & 2 & 2 & 1 \\ 
 \hline
D_{\s_2}    & 2/r &1 & 1 & 1 & 1/2 \\ 
   \hline
\end{tabular}
}
\caption{Root systems and the constant $D_{\s_i}$}
\label{table: Root systems and the constant Di}
\end{table}

The constant $D_{\s_i}$ is computed in Table \ref{table: Root systems and the constant Di} for the infinite families of root systems. 

We now explain the contents of Table \ref{table: Root systems and the constant Di}. The columns are indexed by the five infinite families of root systems: $A_r, B_r, C_r, BC_r$ and $D_r$.

In the first line, we describe the classical choice for the real vector space $V = X(S) \otimes \re$ so that the root system $\Phi$ is described as in \cite[VI, Planches, I-IV]{bourbaki} with simple roots $\t_1, \ldots, \t_r$. In the second line we describe our choice of simple root $\g \in \S$, written in terms of the natural description in coordinates of each root system. The third line computes the constants $n_\g(\a)$ for $\a \in \Phi_{n_\g>0}$. The fourth line describes the linear equation we choose so that the ambient space $X(A^{\g}) \otimes \re$ of the sub-root system $\Phi_{\S \setminus \{ \g\}}$ of $A^{\g}$ sits naturally inside the ambient space $V$ of $\Phi$.

In the next part of Table \ref{table: Root systems and the constant Di}, we first list our choices for $\s_1$ and $\s_2$. We then compute the preliminary constant $C_{\s_i}$ thanks to the equation for $A^\g$ that we computed before. The computation of $\max_{\a \in \Phi_{n_\g > 0}}\{ \frac{n_{\s_i}(\a)}{n_\g(\a)} \} $ requires to first compute $n_{\s_i}(\a)$ for $\a \in \Phi_{n_\g>0}$ (this can be found in \cite[VI, Planches I-IV]{bourbaki}), we only recover the final result in Table \ref{table: Root systems and the constant Di}. The computation of $D_{\s_i}$ amounts to add the two last rows.

\paragraph{Multiplicities}

The multiplicities of the roots in each of the six families listed above are represented in Table \ref{table: Multiplicities of roots} and can be found in \cite[Chapter V, 23]{borel-algebraic-groups}.

\begin{table}
\centering
\begin{tabular}{ |L|L|L|L|L|L|L| } 
\hline
 & \mathrm{SL}_{r+1}(D) &  \mathrm{SU}(h) & \mathrm{SO}(q) & \mathrm{Sp}_{2r}(F) &  \mathrm{SU}(\Tilde{h}) & \mathrm{SO}(\Tilde{q})  \\
 \hline
 e_i - e_j &d^2 & 2 & 1 & 1 & 4 & 4 \\
  \hline
 e_i + e_j & 0 & 2 & 1 & 1 & 4 & 4\\ 
  \hline
 e_i        & 0 & 2(n-2r) & n-2r & 0 & 4(n-2r) & 4(n-2r)\\ 
   \hline
 2e_i & 0 &  1 & 0   & 1 & 3 & 1\\ 
 \hline
\end{tabular}
\caption{Multiplicities of roots in terms of classical presentations}
\label{table: Multiplicities of roots}
\end{table}

We rewrite the inequality appearing in Proposition \ref{Combinatorial-Inequality} as follows:
\begin{equation*}
        D_{\s_i}^{-1} \sum_{\a \in \Phi_{n_{\s_i} > 0} \cap (\Phi_{\S \setminus \{ \g \}})^+} n_{\s_i}(\a) m(\a) > Q_\g.
\end{equation*} 

In Table \ref{Computation of the Hausdorff dimension and the LHS}, we compute the two sides of this inequality using the multiplicities appearing in Table \ref{table: Multiplicities of roots} for the choices of $\g, \s_1$ and $\s_2$ made in Table \ref{table: Root systems and the constant Di}. More precisely, we first compute the dimension $Q_\g = \mathrm{Hausdim}(U_\g, | \cdot |_\g) = \sum_{\a \in \Phi_{n_\g > 0}} n_\g(\a) m(\a)$ and then the term appearing in the left hand side $M_i = \sum_{\a \in \Phi_{n_{\s_i} > 0} \cap (\Phi_{\S \setminus \{ \g \}})^+} n_{\s_i}(\a) m(\a)$ for our choices of roots $\s_1$ and $\s_2$. The line called LHS for $\s_i$ contains the term $D_{\s_i}^{-1} M_i = D_{\s_i}^{-1} \sum_{\a \in \Phi_{n_{\s_i} > 0} \cap (\Phi_{\S \setminus \{ \g \}})^+} n_{\s_i}(\a) m(\a)$. For groups preserving some type of form, $n$ denotes the number of variables of the form and the $F$-rank $r$ of the group coincides with its Witt index. In the line "Inequality for $\s_i$" we recover the conditions on $r$ so that the condition in Proposition \ref{Combinatorial-Inequality} is satisfied for $\s_1$ or $\s_2$. Items having two lines correspond to quantities that change depending on the root system of the group (the only line that does not follow this is the line "Inequality for $\s_i$" where conditions on the rank are slightly more complicated).

\begin{proof}[Proof of theorem \ref{Vanishing in degree 2}]
By looking at Table \ref{Computation of the Hausdorff dimension and the LHS}, we see that the inequality of Proposition \ref{Combinatorial-Inequality} is satisfied for both roots $\s_1$ and $\s_2$ for $r \geq 4$ in all cases, except for $\mathrm{SO}(q)$ where $q$ is a quadratic form in $8$ variables and Witt index 4. For $G = \SL(4, F)$, there is strict inequality for $\s_1$ and equality for $\s_2$. Hence we can apply Theorem \ref{VanishingCriterion} and conclude using Corollary \ref{Hk(G,Lp(G))=Hk-1(R,Lp(R,H1(AN,Lp(AN))))}.
\end{proof}

\begin{table}
\centering
\resizebox{\columnwidth}{!}{
\begin{tabular}{ |L|L|L|L|L|L|L| } 
\hline
 & \mathrm{SL}_{r+1}(D) &  \mathrm{SU}(h) & \mathrm{SO}(q) & \mathrm{Sp}_{2r}(F) &  \mathrm{SU}(\Tilde{h}) & \mathrm{SO}(\Tilde{q})  \\
 \hline
  \multirow{2}{4em}{\textrm{Root system}}  & \multirow{2}{4em}{$A_r$} & C_r \textrm{ if }n = 2r & D_r \textrm{ if }n = 2r & \multirow{2}{4em}{$C_r$} & C_r \textrm{ if }n = 2r & C_r \textrm{ if }n = 2r\\ 
  &  & BC_r \textrm{ if }n > 2r & B_r \textrm{ if }n > 2r &  & BC_r \textrm{ if }n > 2r & BC_r \textrm{ if }n > 2r\\ 
  \hline
 Q_\g &r d^2 & 2n-2 & n-2 & 2r & 4n-2 & 4n-6 \\
  \hline
  \hline
 \multirow{2}{4em}{$D_{\s_1}$}  & \multirow{2}{4em}{$1/r$} & 1/2 & 1/2 & \multirow{2}{4em}{$1/2$} & 1/2 & 1/2\\ 
  &  & 1 & 1 &  & 1 & 1\\ 
  \hline
  \multirow{2}{4em}{ $M_1$}  & \multirow{2}{4em}{$(r-1)d^2$} & (r-1)^2  & (r-1)(r-2)/2 & \multirow{2}{5em}{$r(r-1)/2$} & (r-1)(2r-1) & (r-1)(2r-3)\\ 
   & & 2(r-1)(n-r-1) & (r-1)(n-r-2) &  & 2(r-1)(2n-2r-1) & 2(r-1)(2n - 2r-3)\\  
   \hline
    \textrm{LHS for }\s_1 & r(r-1)d^2 & 2(r-1)(n-r-1) & (r-1)(n-r-2) & r(r-1) & 2(r-1)(2n-2r-1) & 2(r-1)(2n - 2r-3) \\
  \hline
 \multirow{4}{4em}{\textrm{Inequality for } $\s_1$}& r \geq 3  &   r \geq 4 \textrm{ if } n -2r = 0, 1 &  r \geq 5 \textrm{ (equality } r = 4 )  & r \geq 4 &   r \geq 4 \textrm{ if } n = 2r &  r \geq 4 \textrm{ if } n -2r =0 , 1 \\ 
  & (\textrm{equality for} &  r \geq 3 \textrm{ if } n -2r \geq 2  &  \textrm{ if } n  = 2r   & (\textrm{equality } r=3) &   r \geq 3 \textrm{ if } n -2r \geq 1 & r \geq 3 \textrm{ if } n -2r \geq 2 \\ 
& r=2) &  &  r \geq 4 \textrm{ if } n -2r = 1,2  &  &  & \\ 
& &  &  r \geq 3 \textrm{ if } n -2r \geq 3  &  &  & \\
   \hline
   \hline
 \multirow{2}{4em}{$D_{\s_2}$}  & \multirow{2}{4em}{$2/r$} & \multirow{2}{4em}{$1$} & 1/2 & \multirow{2}{4em}{$1$} & \multirow{2}{4em}{$1$} & \multirow{2}{4em}{$1$}\\ 
  &  &  & 1 &  &  & \\ 
  \hline
  \multirow{2}{4em}{$M_2$}  & \multirow{2}{5em}{$2(r-2)d^2$} & \multirow{2}{6em}{$2(n-r)(r-2)$}  & (r-1)(r-2)/2 & \multirow{2}{6em}{$(r-2)(r+1)$}  & \multirow{2}{8em}{$2(r-2)(2n - 2r +1)$}  & \multirow{2}{8em}{$2(r-2)(2n - 2r -1)$} \\ 
   & &  & (r-2)(n-r-1) &  &  & \\  
   \hline
    \textrm{LHS for }\s_2 & r(r-2)d^2 & 2(r-2)(n-r) & (r-2)(n-r-1) & (r-2)(r+1) & 2(r-2)(2n-2r+1) & 2(r-2)(2n - 2r-1) \\
  \hline
 \multirow{3}{4em}{\textrm{Inequality for } $\s_2$}& r \geq 4  &  \multirow{3}{4em}{$r \geq 4$} &  r\geq 5 \textrm{ (equality } r = 4 )   & \multirow{3}{4em}{$r \geq 4$} &  \multirow{3}{4em}{$r \geq 4$}& \multirow{3}{4em}{$r \geq 4$}\\ 
  & (\textrm{equality } r=3) &  & \textrm{if } n = 2r,  &  &  & \\ 
  &  &  &  r\geq 4 \textrm{ otherwise }  &  &  & \\ 
   \hline
\end{tabular}
}

\caption{Computation of $Q_\g$ and $\sum_{\a \in \Phi_{n_{\s_i} > 0} \cap (\Phi_{\S \setminus \{ \g \}})^+} n_{\s_i}(\a) m(\a)$}
\label{Computation of the Hausdorff dimension and the LHS}
\end{table}

\subsection{Uniform vanishing for admissible simple real Lie groups}

In \cite{bourdon-remy-vanishings}, Bourdon and Rémy obtain vanishing of $L^p$-cohomology of many real simple Lie groups in many degrees, for values of $p$ depending on the degree in question. The groups for which their result applies are the simple Lie groups for which there exists some maximal parabolic subgroup such that its solvable radical is isometric to some real hyperbolic space of dimension $d$, for some $d \geq 2$. They call these groups \textit{admissible}. We state their vanishing result for degree 2.

\begin{thm}\label{Uniform vanishing for admissible groups} \cite[Theorem 1.4]{bourdon-remy-vanishings}
Let $G$ be an admissible simple real Lie group. Let $d$ be the dimension of the real hyperbolic space associated to $G$. Then
\begin{equation*}
    H^2_{\mathrm{ct}}(G, L^p(G)) = \{0\} \; \textrm{ for } \; p \leq \frac{d-1}{2}.
\end{equation*}
\end{thm}

Their result gives vanishing of the second $L^p$-cohomology group for small values of $p > 1$. On the other hand, when $F = \re$ or $\co$, Theorem \ref{Vanishing in degree 2} gives vanishing of the second $L^p$-cohomology group for large values of $p$. In this section, we address the question of when these two results combined give vanishing for all $p> 1$.

In combinatorial terms, the admissibility condition amounts to ask that there exists some simple root $\s$ such that the coefficients $n_\s(\a)$ are $1$ for all roots $ \a \in \Phi_{n_\s > 0}$. Such a simple root is called a \textit{good} root. In this case the group $H_\s$ is the real hyperbolic space of dimension $d$, with $d -1 = Q_\s$. Then the previous cited theorem gives vanishing for $p \leq \max \{ Q_\s , \, \s \textrm{ good root} \} / 2$. Theorem \ref{Vanishing in degree 2} gives the desired vanishing for $p > Q_\g$, where $\g$ is our choice of simple root, as in Table \ref{table: Root systems and the constant Di}. Then the condition we need to guarantee that $H^2_{\mathrm{ct}}(G, L^p(G)) = \{0\}$ for all $p > 1$ is
\begin{equation*}
    Q_\g  \leq \max \{ Q_\s , \, \s \textrm{ good root} \} / 2.
\end{equation*}

Our choice of simple root $\g$ was made so that $Q_\g$ is small, in particular in all our choices $Q_\g$ grows linearly with the rank $r$ of the group $G$. The previous inequality has chances to be satisfied for many groups of large rank, since there exists often some good root $\s$ such that $Q_\s$ grows quadratically with $r$. More precisely, such a root exists in the root systems $A_r$, $C_r$ and $D_r$ but not in $B_r$. The following corollary is obtained using the computation of $Q_\g$ present in Table \ref{Computation of the Hausdorff dimension and the LHS} and the computation of $\max \{ Q_\s , \, \s \textrm{ good root} \}$ that can be found in the tables \cite[p. 1319 and 1320]{bourdon-remy-vanishings}.

\begin{cor}
Let $G$ be one of the following admissible simple Lie groups: $\SL_r(\re)$, $\SL_r(\co)$, $\SL_r(\mathbb{H})$, $\mathrm{Sp}_{2r}(\re)$, $\mathrm{Sp}_{2r}(\co)$ with $r \geq 7$, or $\mathrm{SU}_{r,r}(\re)$, $\mathrm{Sp}_{2r,2r}(\re)$, $\mathrm{SO}_{r,r}(\re)$, $\mathrm{SO}_{2r}(\co)$, $\mathrm{SO}^*(4r)$ with $r \geq 8$. Then:
\begin{equation*}
    H^2_{\mathrm{ct}}(G, L^p(G)) = \{0\} \; \textrm{ for all } \; p > 1.
\end{equation*}
\end{cor}

\begin{rem}
The only classical families of admissible simple real Lie groups missing in this corollary are $\mathrm{SO}_{2r+1}(\co)$ and $\mathrm{SO}_{r,n - r}(\re)$ with $n>2r$. These are the admissible groups with (restricted) root system $B_r$. We cannot obtain vanishing for all $p> 1$ for these groups by complementing our results with those from \cite{bourdon-remy-vanishings} because our choice of $\g$ is in fact the only good root in the root system $B_r$ and it is not possible to have $Q_\g \leq Q_\g / 2$.
\end{rem}

\bibliographystyle{amsalpha}
\bibliography{refs.bib}

\noindent Antonio López Neumann \\
Mathematical Institute of the Polish Academy of Sciences (IMPAN), Warsaw \\ 00-656 Warsaw, Poland \\
alopez@impan.pl 

\end{document}